\documentclass[12pt]{article}
\date{February 10, 2013}

\newcommand{\ingraph}[1]{\includegraphics{#1}}     

\usepackage{latexsym}
\usepackage{graphics}
\usepackage{amsmath}
\usepackage{xspace}
\usepackage{amssymb}
\usepackage{psfrag}
\usepackage{epsfig}
\usepackage{amsthm}
\usepackage{pst-all}

\usepackage{color}

\definecolor{Red}{rgb}{1,0,0}
\definecolor{Blue}{rgb}{0,0,1}
\definecolor{Olive}{rgb}{0.41,0.55,0.13}
\definecolor{Yarok}{rgb}{0,0.5,0}
\definecolor{Green}{rgb}{0,1,0}
\definecolor{MGreen}{rgb}{0,0.8,0}
\definecolor{DGreen}{rgb}{0,0.55,0}
\definecolor{Yellow}{rgb}{1,1,0}
\definecolor{Cyan}{rgb}{0,1,1}
\definecolor{Magenta}{rgb}{1,0,1}
\definecolor{Orange}{rgb}{1,.5,0}
\definecolor{Violet}{rgb}{.5,0,.5}
\definecolor{Purple}{rgb}{.75,0,.25}
\definecolor{Brown}{rgb}{.75,.5,.25}
\definecolor{Grey}{rgb}{.5,.5,.5}

\newcommand{\Hom}{\ensuremath{\operatorname{Hom}}\xspace}

\newcommand{\pr}{\mathbb{P}}

\newcommand{\R}{\mathbb{R}}
\newcommand{\Z}{\mathbb{Z}}
\newcommand{\G}{\mathbb{G}}
\newcommand{\F}{\mathbb{F}}
\newcommand{\Hg}{\mathbb{H}}

\newcommand{\Lnorm}{\mathbb{L}}

\newcommand{\M}{{\bf \mathcal{M}}}
\newcommand{\N}{{\bf \mathcal{N}}}

\newcommand{\Dk}{[0,D]^{(k+1)\times k}}
\newcommand{\Dm}{[0,D]^{(m+1)\times m}}

\newtheorem{theorem}{Theorem}
\newtheorem{ex}{Example}
\newtheorem{remark}{Remark}
\newtheorem{lemma}{Lemma}
\newtheorem{conj}{Conjecture}
\newtheorem{prop}{Proposition}
\newtheorem{coro}{Corollary}

\newtheorem{Defi}{Definition}

\newcommand{\ignore}[1]{\relax}

\definecolor{Red}{rgb}{1,0,0}
\definecolor{Blue}{rgb}{0,0,1}
\definecolor{Olive}{rgb}{0.41,0.55,0.13}
\definecolor{Green}{rgb}{0,1,0}
\definecolor{MGreen}{rgb}{0,0.8,0}
\definecolor{DGreen}{rgb}{0,0.55,0}
\definecolor{Yellow}{rgb}{1,1,0}
\definecolor{Cyan}{rgb}{0,1,1}
\definecolor{Magenta}{rgb}{1,0,1}
\definecolor{Orange}{rgb}{1,.5,0}
\definecolor{Violet}{rgb}{.5,0,.5}
\definecolor{Purple}{rgb}{.75,0,.25}
\definecolor{Brown}{rgb}{.75,.5,.25}
\definecolor{Grey}{rgb}{.5,.5,.5}
\definecolor{Pink}{rgb}{1,0,1}
\definecolor{DBrown}{rgb}{.5,.34,.16}
\definecolor{Black}{rgb}{0,0,0}

\author{
{\sf Christian Borgs}\thanks{Microsoft Research New England; e-mail: {\tt borgs@microsoft.com}.}
\and
{\sf Jennifer Chayes}\thanks{Microsoft Research New England; e-mail: {\tt jchayes@microsoft.com}.}
\and
{\sf David Gamarnik}\thanks{MIT; e-mail: {\tt gamarnik@mit.edu}.Research supported  by the NSF grants CMMI-1031332.}
}

\begin{document}

\title{Convergent sequences of sparse graphs: A large deviations approach}

\maketitle

\begin{abstract}

In this paper we introduce a new notion of convergence of
sparse graphs which we call Large Deviations or LD-convergence
and which is based on the theory of  large deviations. The
notion is introduced by ''decorating'' the nodes of the graph
with random uniform i.i.d. weights and constructing random
measures on $[0,1]$ and $[0,1]^2$ based on the decoration of
nodes and edges. A graph sequence is defined to be converging
if the corresponding sequence of random measures satisfies the
Large Deviations Principle with respect to the topology of weak
convergence on bounded measures on $[0,1]^d, d=1,2$. We then establish that
LD-convergence implies several previous notions of convergence,
namely so-called right-convergence, left-convergence, and
partition-convergence. The corresponding large deviation rate
function can be interpreted as the limit object of the sparse
graph sequence. In particular, we can express the limiting free
energies in terms of this limit object.


Finally, we establish several previously unknown relationships
between the formerly defined notions of convergence. In
particular, we show that partition-convergence does not imply
left or right-convergence, and that right-convergence does not
imply partition-convergence.
\end{abstract}

\newpage
\section{Introduction}
The theory of convergent graph sequences for dense graphs is by
now a well-developed subject, with many {\it a priori} distinct
notions of convergence proved to be equivalent.  For sparse
graph sequences (which in this paper we take to be sequences
with bounded maximal degree), much less is known. While several
of the notions defined in the
 context of dense graphs
were generalized to the sparse setting, and some other
interesting notions were introduced in the literature, the
equivalence between the different notions was either unknown or
known not to hold.  In some cases, it was not even known whether
one is stronger than the other or not.  In this paper, we
introduce a new natural notion of convergence for sparse
graphs, which we call Large Deviations Convergence, and  relate
it to other notions of convergence.  We hope that this new
notion will also ultimately allow us to relate many of the
previous notions to each other.

Let $\G_n$ be a sequence graphs such that the number of
vertices in $\G_n$ goes to infinity.  The question we address
here is the notion of convergence and limit of such a sequence.

For dense graphs, i.e., graphs for which the average degree
grows like the number of vertices, this question is by
now well understood
~\cite{BorgsChayesCountingGraphHom},\cite{BorgsChayesEtAlGraphLimitsI},\cite{BorgsChayesEtAlGraphLimitsII}.
These works introduced and showed the equivalence of
various notions of convergence: {\it
left-convergence}, defined in terms of homomorphisms from a
small graph $\F$ into $\G_n$; {\it right-convergence}, defined
in terms of homomorphisms from $\G_n$ into a weighted graph
$\mathbb H$ with strictly positive edge weights;  convergence
in terms of the so-called {\it cut metric}; and several other
notions, including convergence of the set of {\it quotients} of
the graphs $\G_n$, a notion which will play a role in this
paper as well. The existence and properties of the limit object
were established in \cite{LovaszSzegedy}, and its uniqueness
was proved in \cite{borgs2010moments}. Some lovely follow up
work on the dense case gave  alternative proofs in terms of
exchangeable random variables~\cite{diaconis2007graph}, and
provided applications using large
deviations~\cite{ChatterjeeVaradhan}.

For sparse graph sequences much less is known. In
\cite{BorgsChayesKahnLovasz}, \emph{left-convergence} was
defined by requiring that the limit
\begin{equation}
\label{L-convergent}
t(\F)=\lim_{n\to \infty}\frac 1{|V(\G_n)|} \hom(\F,\G_n)
\end{equation}
exists for every connected, simple graph $\F$, where
$\hom(\F,\G)$ is used to denote the number of homomorphisms
from $\F$ into  $\G$. Here and in the rest of this paper
$V(\G)$ denotes the
vertex set of $\G$. Also $E(\G)$ denotes the edge set of $\G$
and for $u,v\in V(\G)$ we write $u\sim v$ if $(u,v)\in E(\G)$.
It is easy to see that for sequences of graphs with bounded
degrees, this notion is
equivalent to an earlier
notion of convergence, the notion of \emph{local convergence}
introduced by Benjamini and Schramm~\cite{benjamini_schramm}.

To define right-convergence, one again considers homomorphisms,
but now from $\G_n$ into a small ``target'' weighted graph
$\Hg$.  Each homomorphism is equipped with the weight induced
by $\Hg$.
Since the total weight of all homomorphisms denoted by
$\hom(\G_n,\Hg)$ is at most exponential in $|V(\G_n)|$,  it is
natural to consider the sequence
\begin{align}\label{eq:rightlimit}
f(\G_n,\Hg)=-\frac 1{|V(\G_n)|}\log \hom(\G_n,\Hg),
\end{align}
which in the terminology of statistical physics is the sequence
of free energies. In~\cite{BorgsChayesKahnLovasz} a sequence
$\G_n$  was defined to be right-convergent with respect to a
given graph $\Hg$ if the free energies $f(\G_n,\Hg)$ converge
as $n\to\infty$. It was shown  that a sequence $\G_n$ that is
right-convergent on all simple graphs $\Hg$
is also left-convergent.  But the converse was  shown not to be
true. The counterexample is simple: Let $C_n$ be the cycle on
$n$ nodes, and let $\Hg$ be a bipartite graph.  Then
$\hom(C_n,\Hg)=0$ for odd $n$, while $\hom(C_n,\Hg)\geq 1$ for
even $n$.

This example, however, relies crucially on the fact that $\Hg$
encodes so-called \emph{ hard-core} constraints, meaning that
some of the maps $\phi:V(\G_n)\to V(\Hg)$ do not contribute to
$\hom(\G_n,\Hg)$.  As we discuss in Section 4, such hard-core
constraints are not very natural in many respects, one of them
being that the existence or value of the limit of the sequence
\eqref{eq:rightlimit} can change if we remove a sub-linear
number of edges from $\G_n$ (in the case of $C_n$, just
removing a single edge makes the sequence convergent, as can be
seen by a simple sub-additivity argument).  Hard-core
constraints are also not very natural from a point of view of
applications in physics, where hard-core constraints  usually
represent an idealized limit like the zero-temperature limit,
which is a mathematical idealization of something that in
reality can never be realized.

We therefore propose to \emph{define} right-convergence by
restricting ourselves to so-called soft-core graphs, i.e.,
graphs $\Hg$ with  strictly positive edge weights. In other
words, we {define} $\G_n$ to be \emph{right-convergent} iff
the free energy defined in
 \eqref{eq:rightlimit} has a limit for
every soft-core graph $\Hg$.  This raises the question of how
to relate the existence of a limit for models with hard-core
constraints to those on soft-core graphs, a question we will
address in Section~\ref{subsection:right-convergence} below,
but we do not make this part of the definition of
right-convergence. Given our relaxed definition of the
right-convergence,  the question then remains whether
left-convergence possibly implies right-convergence. The answer
turns out to be still negative, and involves an example already
considered in \cite{BorgsChayesKahnLovasz}, see
Section~\ref{sec:left-not-right} below.

Next let us turn to the notion of partition-convergence of
sparse graphs, a notion which was introduced by Bollobas and
Riordan \cite{BollobasRiordanMetrics}, motivated by a similar
notion for dense graphs from
~\cite{BorgsChayesEtAlGraphLimitsI},
\cite{BorgsChayesEtAlGraphLimitsII}, where it was called
convergence of quotients. We start by introducing quotients for
sparse graphs. Fix a graph $\G$ and let
$\sigma=(V_1,\dots,V_k)$ be a partition of $V(\G)$ into
disjoint subsets (where some of the $V_i$'s may be empty). We
think of this partition as a (non-proper) coloring
$\sigma:V(\G)\rightarrow [k]$ with $V_i=\sigma^{-1}(i), 1\le
i\le k$. The partition $\sigma$ then induces a weighted graph
$\F=\G/\sigma$ on $[k]=\{1,\dots,k\}$ (called a $k$-quotient)
by setting
\begin{equation}
\label{k-quotient}
x_i(\sigma)=\frac  {|V_i|} {|V|}
\qquad\text{and}\qquad
X_{ij}(\sigma)=\frac {1} {|V|} \left|\{u\in V_i, v\in V_j\colon u\sim v\}\right|.
\end{equation}
Here $x(\sigma)=(x_i(\sigma),~1\le i\le k)$ and
$X(\sigma)=(X_{i,j}(\sigma),~1\le i,j\le k)$ are the set of
node and edge weights of $\F$, respectively. Then for a given
$k$, the set of all possible $k$-quotients
$(x(\sigma),X(\sigma))$ is a discrete subset ${\mathcal
S}_k(\G)\subset\R_+^{(k+1)k}$. Henceforth  $\R (\R_+)$
denotes the set of all (non-negative) reals, and $\Z
(\Z_+)$ denotes the set of all (non-negative) integers.
Motivated by the corresponding notion from dense graphs, one
might want to study whether the sequence of $k$-quotients
${\mathcal S}_k(\G_n)$ converges to a limiting set $\mathcal
S_k\subset\R_+^{(k+1)k}$ with respect to some appropriate
metric on the subsets of $\R_+^{(k+1)k}$. This leads to the
notion of partition-convergence, a notion introduced
in~\cite{BollobasRiordanMetrics}. Since it is not immediate
whether this notion of convergence implies, for example,
left-convergence (in fact, one of the results in this paper is
that it does not), Bollobas and
Riordan~\cite{BollobasRiordanMetrics} also introduced a
stronger notion of colored-neighborhood-convergence. This
notion was further studied by Hatami, Lov{\'a}sz and
Szegedi~\cite{HatamiLovaszSzegedy}. An implication of the
results of these
and earlier papers (details below) is that
colored-neighborhood-convergence is strictly stronger than the notion of both
partition- and left-convergence, but it is not known whether it
implies right-convergence.

We now discuss the main contribution of this paper: the
definition of convergent sparse graph sequences using the
formalism of large deviations theory. We define a graph
sequence to be Large Deviations (LD)-convergent if for every
$k$, the weighted factor graphs $\F=G/\sigma$ defined above
viewed as a vector $(x_i(\F), 1\le i\le k)$ and matrix
$(X_{i,j}(\F), 1\le i,j\le k)$, satisfy the large deviation
principle in $\R^k$ and $\R^{k\times k}$, when
the $k$-partition $\sigma$ is chosen uniformly at random.%
\footnote{Some necessary background on Large Deviations
Principle and weak convergence of measures will be provided in
Section~\ref{section:LDconvergence}.} Intuitively, the large
deviations rate associated with a given graph $\F$ provides the
limiting exponent for the number colorings $\sigma$ such that
the corresponding factor graphs $\G/\sigma$ is approximately
$\F$. It turns out that the large deviations rates provide
enough information to ''read off'' the limiting partition sets
$\mathcal{S}_k$. Those are obtained as partitions with finite
large deviations rate. Similarly, one can ''read off'' the
limits of free energies (\ref{eq:rightlimit}). As a
consequence, the LD-convergence implies both the partition and
right-convergence.

The deficiency of the definition  above is that it requires
that the large deviations principle holds for a infinite
collection of probability measures associated with the choice
of $k$. It turns out that there is a more elegant unifying way
to introduce  LD-convergence  by defining just one large
deviations rate, but on the space of random measures rather
than the space of random weighted graphs $\G_n/\sigma$. This is
done as follows. Given a sequence of sparse graphs $\G_n$,
suppose the nodes of each graph are equipped with values
$\left(\sigma(u),  u\in V(\G)\right)$ chosen independently
uniformly at random from $[0,1]$. We may think of these values
as real valued ''colors''. From these values we construct a
one-dimensional measure $\rho_n=\sum_{u\in V(\G)}
|V(\G_n)|^{-1}\delta(\sigma(u))$ on $[0,1]$ and a
two-dimensional measure $\mu_n=\sum_{u,v\colon u\sim v}
|V(\G_n)|^{-1}\delta(\sigma(u),\sigma(v))$ on $[0,1]^2$. Here
$\delta(x)$ is a measure with unit mass at $x$ and zero
elsewhere. As such we obtain a sequence of measures
$(\rho_n,\mu_n)$. We define the graph sequence to be
LD-convergent if this sequence of random measures satisfy the
Large Deviations Principle on the space of measures on
$[0,1]^d, d=1,2$ equipped with the weak convergence topology.
One of our first result is establishing the equivalence of two
definitions of the LD-convergence. To distinguish the two, the
former mode of convergence is referred to as
$k$-LD-convergence, and the latter simply as LD-convergence.

Our most important result concerning the new definition of
graph convergence is that LD-convergence implies
right (and therefore left) as well as partition-convergence. We
conjecture that it also implies
colored-neighborhood-convergence but we do not have a proof at
this time. Finally, we conjecture that LD-convergence holds for
a sequence of random $D$-regular graphs with high probability
(w.h.p.), but at the present stage we are very far from proving
this conjecture and we discuss important implications to the
theory of spin glasses should this convergence be established.

Let us finally return to a question already alluded to when we
defined right-convergence. From several points of view, it
seems natural to consider two sequence $\G_n$ and $\tilde \G_n$
on the same vertex set $V_n=V(\G_n)=V(\tilde\G_n)$ to be equivalent
if the edge sets $E(\G_n)$ and $E(\tilde\G_n)$ differ on a set
of $o(|V_n|)$ edges. One natural condition one might require
of all definitions of convergence is the condition that
convergence of $\G_n$ implies that of $\tilde\G_n$ and vice
versa, whenever $\G_n$ and $\tilde\G_n$ are equivalent. We
show that this holds for all notions of convergence considered
in the paper, namely left-convergence, our modified version of
right-convergence, partition-convergence,
colored-neighborhood-convergence, and LD-convergence.  We also
show that for all right-convergent sequences $\G_n$, one can
find an equivalent sequence $\tilde\G_n$ such that on
$\tilde\G_n$, the limiting free energy
$f(\Hg)=\lim_{n\to\infty} f(\tilde\G_n,\Hg)$ exists even when
$\Hg$ contains hard-core constraints, see
Theorem~\ref{theorem:RightConvergenceEquivalence} below.  Our
alternative definition of right-convergence can therefore be
reexpressed by the condition that there exists a sequence
$\tilde\G_n$ equivalent to $\G_n$ such that the free
energies $f(\tilde\G_n,\Hg)$ converge for all weighted graphs
$\Hg$.

The organization of this paper is as follows. In the next
section we introduce notations and definitions and provide the
necessary background on the large deviations theory. In
Section~\ref{section:LDconvergence} we introduce
LD-convergence, establish some basic properties and consider
two examples of LD-convergent graph sequences. Other notions of
convergence are discussed in
Section~\ref{section:OtherConvergence}. In the same section we
prove Theorem~\ref{theorem:RightConvergenceEquivalence}
relating free energy limits with respect to hard-core and soft
models, discussed above. The relationship between different
notions of convergence is considered in
Section~\ref{section:ConvergenceRelations}. The main result of
this section is Theorem~\ref{theorem:ConvergenceRelations},
which establishes relationships between different notions of
convergence. Figures~\ref{figure:ConvergenceRelations} and
\ref{figure:ConvergenceRelationsColorNeighb} are provided to
illustrate the relations between the notions of convergence, both
those known earlier and those established in
this paper. In the last
Section~\ref{sec:discussion} we compare
the expressions for the limiting free energy to those in the dense
case (pointing out that the
large-deviations rate function plays a role quite similar to the role
the limiting graphon played in the dense case) and
discuss  some open
questions.

\section{Notations and definitions}
\label{sec:notation}

For  the convenience of the reader, we repeat
some of the notations already
used in the introduction. We consider in this paper finite
simple undirected graphs $\G$ with node and edge sets denoted
by $V(\G)$ and $E(\G)$ respectively. We write $u\sim v$ for
nodes $u$ and $v$ when $(u,v)\in E(\G)$. For every node $u\in
V(\G)$, $\mathcal{N}(u)$ denotes the set of neighbors of $u$,
namely all $v\in V(\G)$ such that $u\sim v$. Sometimes we will
write $\mathcal{N}_{\G}(u)$  in order to emphasize the
underlying graph. The number $\Delta_\G(u)$ of vertices in
$\mathcal{N}_{\G}(u)$ is called the degree of $u$ in $\G$. We
use $\Delta_\G$ to denote the maximum degree $\max_u
\Delta_G(u)$.

A path of length $r$ is a sequence of nodes $u_{0},\ldots,u_r$
such that $u_i\sim u_{i+1}$ for all $i={0},1,2,\ldots,r-1$. The
distance between nodes $u$ and $v$ is the length of the
shortest path $u_0,\ldots,u_r$ with $u_0=u$ and $u_r=v$. The
distance is assumed to be infinite if $u$ and $v$ belong to
different connected components of $\G$. Given $u$ and $r$, let
$B(u,r)=B_{\G}(u,r)$ be the subgraph of $\G$ induced by nodes
with distance at most $r$ from $u$. Given a set of nodes
$A\subset V$ and $r\ge 1$, $B(A,r)$ is the graph induced by the
set of nodes $v$ with distance at most $r$ from some node $u\in
A$.

Given two graphs $\G,\Hg$ a map $\sigma:V(\G)\rightarrow
V(\Hg)$ is a graph homomorphism if   $(u,v)\in E(\G)$ implies
$(\sigma(u),\sigma(v))\in E(\Hg)$ for every $u,v\in V(\G)$.
Namely, $\sigma$ maps edges into edges. $\sigma:\G_1\rightarrow
\G_2$ is called a graph isomorphism if it is a bijection and
$(u,v)\in E(\G_1)$ if and only if $(\sigma (u), \sigma(v))\in
E(\G_2)$, i.e., $\sigma$ is an isomorphism which maps edges
into edges, and non-edges into non-edges. Two graphs are called
isomorphic if there is a graph isomorphism mapping them into
each other.
 Two rooted graphs $(\G_1,x)$ and $(\G_2,y)$ are called isomorphic
if there is an isomorphism $\sigma:\G_1\rightarrow \G_2$ such
that $\sigma(x)=y$.

The edit distance between two graphs $\G$ and $\tilde\G$ with
the same number of vertices is defined as
\[
\delta_{\text{edit}}(\G,\tilde\G)=
{\min}_{\G'}\Bigl(|E(\G')\setminus E(\tilde\G)| + |E(\tilde\G)\setminus E(\G')|\Bigr)
\]
where the {minimum} goes over all graph $\G'$ isomorphic to
$\G$.

Given two graphs $\G,\Hg$, we use $\Hom(\G,\Hg)$ to denote the
set of all homomorphisms from $\G$ to $\Hg$, and
$\hom(\G,\Hg)=|\Hom(\G,\Hg)|$ to denote the number of
homomorphisms from $\G$ to $\Hg$.  If $\Hg$ is weighted, with
node weights given by a vector $\left(\alpha_i, i\in
V(\Hg)\right)$ and edge weights given by a symmetric matrix
$\left(A_{ij}, (i,j)\in E(\Hg)\right)$ (which we denote by
$\alpha=\alpha(\Hg)$ and $A=A(\Hg)$ respectively), then
$\hom(\G,\Hg)$ is defined by
\begin{align}\label{eq:WeightedHom}
\hom(\G,\Hg)=\sum_{\sigma:V(\G)\rightarrow V(\Hg)}\prod_{u\in V(\G)}\alpha_{\sigma(u)}\prod_{(u,v)\in E(\G)}A_{\sigma(u),\sigma(v)}.
\end{align}
Clearly, if $\alpha$ is a vector of ones and $A$ is a matrix
consisting of zeros and ones with zeros on the diagonal, then the
definition reduces to the unweighted case where $A$ denotes the
adjacency matrix of a simple undirected graph $\Hg$.

A graph sequence $\G_n=(V(\G_n),E(\G_n)), n\ge 1$ is defined to
be sparse if $\Delta_{\G_n}\le D$ for some finite $D$ for all
$n$. When dealing with graph sequences we  write $V_n$ and
$E_n$ instead of $V(\G_n)$ and $E(\G_n)$, respectively. Two
graph sequences $\G_n$ and $\tilde\G_n$ are defined to be
equivalent, in which case we write $\G_n\sim \tilde\G_n$, if
$\tilde\G_n$ has the same number of nodes as $\G_n$ and
$\delta_{\text{edit}}(\G_n,\tilde\G_n)=o(|V(\G_n)|)$ as
$n\rightarrow\infty$. Observe that $\sim$ defines an
equivalency relationship on the set of sequences of graphs.

Next we review some concepts from large deviations theory. Given a
metric space $S$ equipped with metric $d$, a sequence of
probability measures $\pr_n$ on Borel sets of $S$ is said to
satisfy the Large Deviations Principle (LDP) with rate
$\theta_n>0, \lim_n\theta_n=\infty$, if there exists a lower
semi-continuous function $I:S\rightarrow \R_+\cup \{\infty\}$
such that for every (Borel) set $A\subset S$
\begin{align}\label{eq:LDprinciple}
-\inf_{x\in A^o}I(x)\le \liminf_n {\log\pr_{n}(A^o)\over \theta_n}\le \limsup_n {\log \pr_{n}(\bar A)\over \theta_n} \le -\inf_{x\in \bar A}I(x),
\end{align}
where $A^o$ and $\bar A$ denote the interior and the closure of
the set $A$, respectively. In this case, we say that $(\pr_n)$
obeys the LDP with rate function $I$. Typically $\theta_n=n$
is used in most cases, but in our case we will be considering
the normalization $\theta_n=|V(\G_n)|$ corresponding to some
sequence of graphs $\G_n$, and it is convenient not to assume
that $n$ is the number of nodes in $\G_n$.

The definition above immediately implies that that the rate
function $I$ is uniquely recovered as
\begin{align}\label{eq:Identity1}
I(x)=-\lim_{\epsilon\rightarrow 0}\liminf_n {\log\pr_{n}(B(x,\epsilon))\over \theta_n}=
-\lim_{\epsilon\rightarrow 0}\limsup_n {\log\pr_{n}(B(x,\epsilon))\over \theta_n},
\end{align}
for every $x\in S$, where $B(x,\epsilon)=\{y\in S:d(x,y)\le
\epsilon\}$. Indeed, from lower semi-continuity we have
\begin{align}\label{eq:Ikepsilon}
I(x)=\lim_{\epsilon\rightarrow 0}\inf_{y\in B^o(x,\epsilon)}I(y)=\lim_{\epsilon\rightarrow 0}\inf_{y\in \bar B(x,\epsilon)}I(y),
\end{align}
from which the claim follows. In fact the existence and the
equality of double limits in (\ref{eq:Identity1}) is also
sufficient for the LDP to hold when $S$ is compact. We provide
a proof here for completeness.

\begin{prop}\label{prop:LDReverse}
Given a sequence of probability measures $\pr_n$ on a compact
metric space $S$, and given a sequence $\theta_n>0,
\lim_n\theta_n=\infty$, suppose the limits in the second and
third term in (\ref{eq:Identity1}) exist and are equal for
every $x$. Let $I(x)$ be defined as the negative of these terms
for every $x\in S$. Then the LDP holds with rate function $I$
and normalization $\theta_n$.
\end{prop}
\begin{proof}
We first establish that $I$ is lower semi-continuous. Fix
$x,x_m\in S$ such that $x_m\rightarrow x$. For every
$\epsilon$, find $m_0$ such that $d(x,x_m)\le \epsilon/2$ for
$m\ge m_0$. We have for all $m\ge m_0$,
\begin{align*}
-I(x_m)\le \limsup_n{1\over \theta_n}\log\pr_{n}\left(B(x_m,\epsilon/2)\right)\le \limsup_n{1\over \theta_n}\log\pr_{n}\left(B(x,\epsilon)\right),
\end{align*}
implying
\begin{align*}
\limsup_m (-I(x_m))\le \lim_{\epsilon\rightarrow 0}\limsup_n{1\over \theta_n}\log\pr_{n}\left(B(x,\epsilon)\right)=-I(x).
\end{align*}
Therefore $I$ is lower semi-continuous.

It remains to verify (\ref{eq:LDprinciple}). Fix an arbitrary
closed (and therefore compact) $A\subset S$, and let
$\delta>0$.  Then
\begin{align*}
\limsup_n {1\over \theta_n}\log\pr_{n}(A)\le
{1\over \theta_n}\log\pr_{n}(A) +\delta
\end{align*}
for an infinite number of values for $n$.  Let $N_0$ be the set
of $n$ for which this holds. Further, given $m\in \Z_+$, we can
find $x_1,\ldots,x_{N(m)} \in A$ such that $A\subset \cup_{i\le
N(m)}B(x_i,1/m)$. Let
$i(m,n)\le N(m)$
be the index corresponding to a largest value of
$\pr_{n}(B(x_i,1/m))$. There exists an $i(m)$ such that
$i(m,n)=i(m)$ for an infinite number of values $n\in N_0$.
Denote the set of $n$ for which this holds by $N_m$, and assume
without loss of generality that $\theta_n^{-1}\log N(m)\leq
\delta$ for all $n\in N_m$.  For such $n$, we then have
\begin{align*}
\pr_{n}(A)\le \sum_{i\le N(m)}\pr_{n}(B(x_i,1/m))\le N(m)\pr_{n}(B(x_{i(m)},1/m))
\end{align*}
and hence
\begin{align*}
\limsup_n {1\over \theta_n}\log\pr_{n}(A)\le 2\delta+
{1\over \theta_n}\log\pr_{n}(B(x_{i(m)},1/m)).
\end{align*}

Find any limit point $x\in A$ of $x_{i(m)}$ as
$m\rightarrow\infty$. Fix $\epsilon>0$. Find $m_1$ such that
$B(x_{i(m_1)},1/m_1)\subset B(x,\epsilon)$, and hence
\begin{align*}
\limsup_n {1\over \theta_n}\log\pr_{n}(A)\le 2\delta+
{1\over \theta_n}\log\pr_{n}(B(x,\epsilon))
\end{align*}
whenever $n\in N_{m_1}$.  Since $N_{m_1}$ contains infinitely
many $n\in \Z_+$, we conclude that
\begin{align*}
\limsup_n {1\over \theta_n}\log\pr_{n}(A)\le 2\delta+
\limsup_n {1\over \theta_n}\log\pr_{n}(B(x,\epsilon)).
\end{align*}
Since this holds for every $\delta$ and  $\epsilon$, we may
apply (\ref{eq:Identity1}) to obtain
\begin{align*}
\limsup_n {1\over \theta_n}\log\pr_{n}(A)\le -I(x)\le -\inf_{y\in A}I(y),
\end{align*}
thus verifying the upper bound in
(\ref{eq:LDprinciple}).

For the lower bound, suppose $A\subset S$ is open. Fix
$\epsilon>0$ and find $x\in A$  such that $I(x)\le \inf_{y\in
A}I(y)+\epsilon$. Find $\delta$ such that $B(x,\delta)\subset
A$. We have
\begin{align*}
\liminf_n {1\over \theta_n}\log\pr_{n}(A)&\ge \liminf_n {1\over \theta_n}\log\pr_{n}\left(B(x,\delta)\right)\\
&\ge -I(x)\\
&\ge -\inf_{y\in A}I(y)-\epsilon.
\end{align*}
Since $\epsilon$ is arbitrary, this proves the lower bound.
\end{proof}

We finish this section with some basic notions of weak
convergence of measures. Let $\M^d$ denote the space of Borel
measures  on $[0,1]^d$ bounded by some constant $D>0$. The
space $\M^d$ is equipped with the topology of weak convergence
which can be realized using the so-called Prokhorov metric
denoted by $d$ defined as follows (see~\cite{Billingsley} for
details). For any two measures $\mu,\nu\in \M^d$, their
distance $d(\mu,\nu)$ is the smallest $\tau$ such that for every
measurable $A\subset [0,1]^d$ we have $\mu(A)\le
\nu(A^\tau)+\tau$ and $\nu(A)\le \mu(A^\tau)+\tau$. Here
$A^\tau$ is the set of points with distance at most $\tau$ from
$A$ in the $\|\cdot\|_\infty$ norm on $[0,1]^d$ (although the
choice of norm on $[0,1]^d$ is not relevant). It is known that
that the metric space $\M^d$ is compact with respect to the
topology of weak convergence. We will focus in this paper on
the product $\M^1\times\M^2$ which we equip with the following
metric (which we also denote by $d$ with some abuse of
notation): for every two pairs of measures
$(\rho_j,\mu_j)\in\M^1\times\M^2,~j=1,2$, the distance is
$d\left((\rho_1,\mu_1),(\rho_2,\mu_2)\right)=\max\left(d(\rho_1,\rho_2),d(\mu_1,\mu_2)\right)$,
where $d$ is the metric on $\M^1$ and $\M^2$.

\section{LD-convergence of sparse graphs}\label{section:LDconvergence}
\subsection{Definition and basic properties}\label{subsection:DefProperties}
Given a graph $\G=(V,E)$ and any mapping $\sigma:V\rightarrow
[0,1]$, construct $(\rho(\sigma),\mu(\sigma))\in \M^1\times
\M^2$ as follows. For every $u\in V$ we put mass $1/|V|$ on
$\sigma(u)\in [0,1]$. This defines $|V|$ points on $[0,1]$ with
the total mass $1$. The resulting measure is denoted by
$\rho(\sigma)\in\M^1$. Similarly, we put mass $1/|V|$ on
$(\sigma(u),\sigma(v))\in [0,1]^2$ and
$(\sigma(v),\sigma(u))\in [0,1]^2$ for every edge $(u,v)\in E$.
The resulting measure is denoted by $\mu(\sigma)\in \M^2$.
Formally
\begin{equation}
\label{rho-sigma}
\begin{aligned}
\rho(\sigma)&=\sum_{u\in V}|V|^{-1}\delta_{\sigma(u)}, \\
\mu(\sigma)&=\sum_{(u,v)\in E}|V|^{-1}\delta_{\sigma(u),\sigma(v)},
\end{aligned}
\end{equation}
where the sum goes over oriented edges, i.e., ordered pairs
$(u,v)$ such that  $u$ is connected to $v$ by an edge in
$E$, and $\delta_x$ denotes a unit mass measure on $x\in\R^d$.
For example, consider a graph with $4$ nodes $1,2,3,4$ and edges
$(1,2),(1,3),(2,3),(3,4)$. Suppose a realization of $\sigma$ is
$\sigma(1)=.4,\sigma(2)=.8,\sigma(3)=.15,\sigma(4)=.5$, as
shown on Figure~\ref{figure:graph}. The corresponding
$\rho(\sigma)$ and $\mu(\sigma)$ are illustrated on
Figure~\ref{figure:rhomu}.

\begin{figure}
\begin{center}
\scalebox{.4}{\ingraph{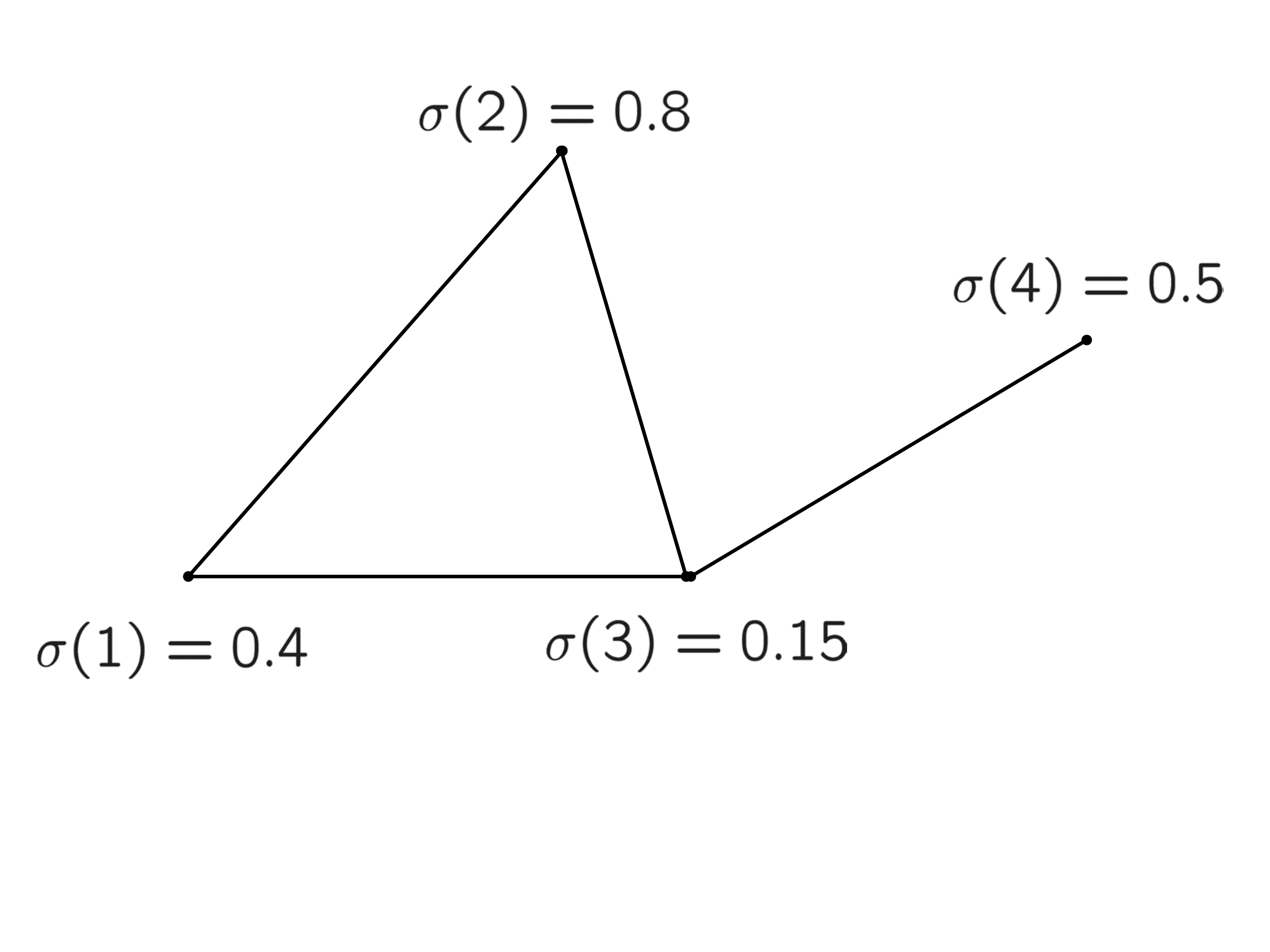}}
\vskip -48pt
\caption{A four node graph decorated with real valued colors}\label{figure:graph}
\end{center}
\end{figure}

\begin{figure}
\begin{center}
\scalebox{.5}{\ingraph{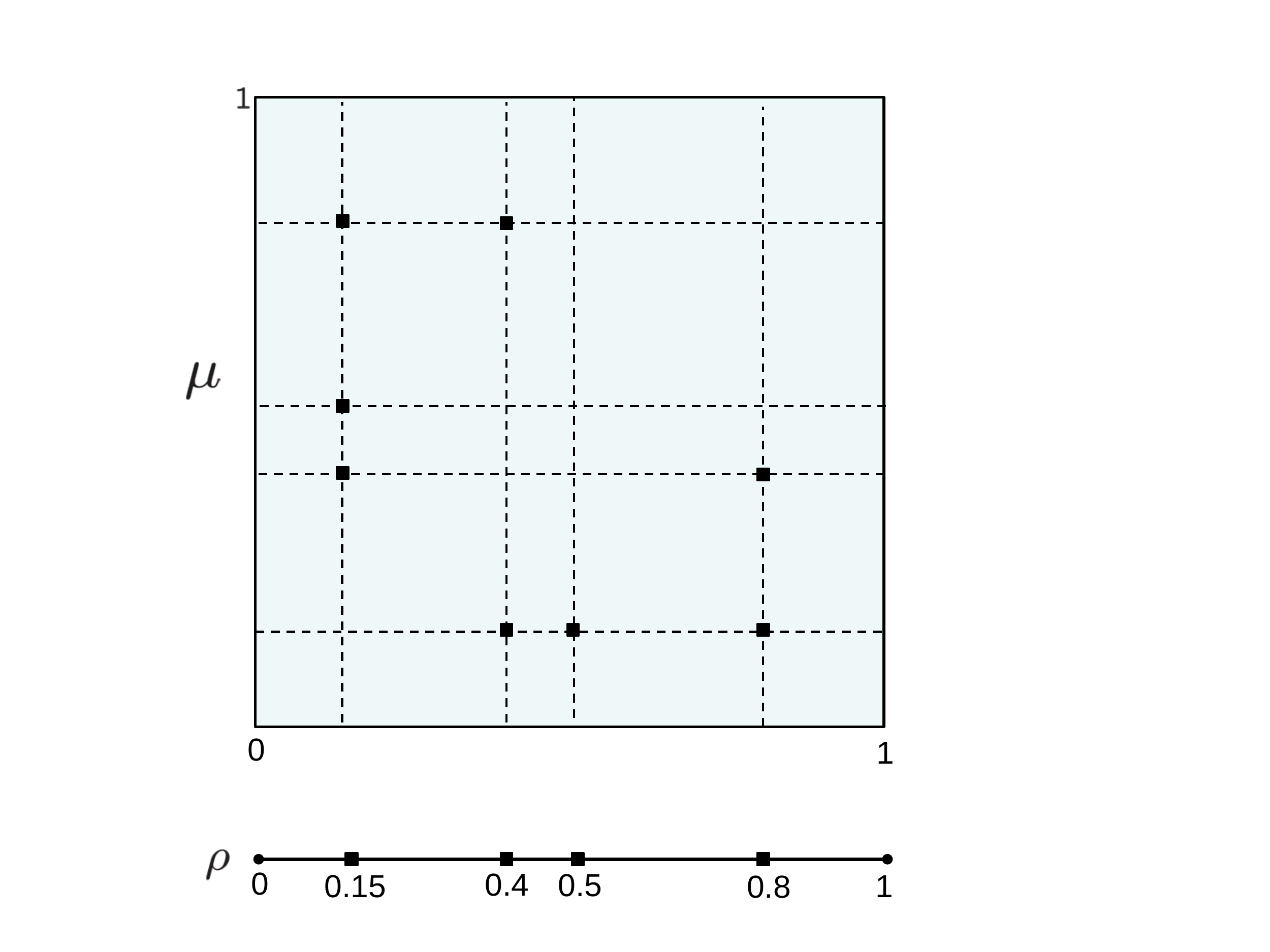}}
\caption{Measures $\rho,\mu$ associated with a four node graph with four edges}\label{figure:rhomu}
\end{center}
\end{figure}

Now given a sparse graph sequence $\G_n=(V_n,E_n)$ with a
uniform degree bound $\max_n \Delta_{\G_n}\le D$, let
$\sigma_n:V_n\rightarrow [0,1]$ be chosen  independently
uniformly at random with respect to the nodes $u\in V_{n}$.
Then we obtain a random element
$(\rho(\sigma_n),\mu(\sigma_n))\in \M^1\times \M^2$. We will
use notations $\rho_n$ and $\mu_n$ instead of $\rho(\sigma_n)$
and $\mu(\sigma_n)$. Observe that $\rho_n([0,1])=1$ and $0\le
\mu_n([0,1]^2)\le D$. In particular, $\mu_n([0,1]^2)=0$ if and
only if $\G_n$ is an empty graph. We now introduce the notion
of LD-convergence. In this definition and the
Definition~\ref{definition:GraphConvergenceLDk} below we assume
$\theta_n=|V_n|$.

\begin{Defi}\label{definition:GraphConvergenceLDGeneral}
A graph sequence $\G_n$ is defined to be LD-convergent if the
sequence $(\rho(\sigma_n),\mu(\sigma_n))$ satisfies LDP in the
metric space $\M^1\times \M^2$.
\end{Defi}

Given a sparse graph sequence $\G_n$ with degree bounded by $D$
and a positive integer $k$, let $\sigma_n:V_n\rightarrow [k]$
be chosen uniformly at random. Consider the corresponding
random weighted  $k$-quotients $\F_n=\G_n/\sigma_n$ defined via
(\ref{k-quotient}) with random node weights
$x(\sigma_n)=(x_i(\sigma_n), 1\le i\le k)$ and random edge
weights $X(\sigma_n)=(X_{ij}(\sigma_n), 1\le i,j\le k)$. Note
that $\sum_i x_i(\sigma_n)=1$ and $\sum_{i,j}
x_{ij}(\sigma_n)\le D$. Assuming without  loss of generality
that $D\ge 1$, it it will be convenient to view
$(x(\sigma_n),X(\sigma_n))$ as a $(k+1)\times k$ matrix with
elements bounded by $D$, namely an element in $\Dk$. We
consider $\Dk$ equipped with the $\Lnorm_\infty$ norm. That is the
distance between $(x,X)$ and $(y,Y)$ is  the maximum of $\max_{1\le
i\le k}|x_i-y_i|$ and $\max_{1\le i,j\le k}|X_{i,j}-Y_{i,j}|$.

\begin{Defi}\label{definition:GraphConvergenceLDk}
A graph sequence $\G_n$ is defined to be $k$-LD-convergent if
for every $k$ the sequence $(x(\sigma_n),X(\sigma_n))$
satisfies the LDP in $\Dk$ with some rate function
$I_k:\Dk\rightarrow \R_+\cup \{\infty\}$.
\end{Defi}

The intuition behind this definition is as follows. Given
$(x,X)$ suppose there exists a $k$-coloring of a graph $G_n$
such that the corresponding $k$-quotient weighted graph is
''approximately'' equal to $(x,X)$ to within some additive
error $\epsilon$. Then in fact there are exponentially in
$|V_n|$ many $k$-colorings of $\G_n$ which achieve ''nearly''
the same quotient graph up to say an $2\epsilon$ additive error
(by arbitrarily recoloring $\epsilon |V_n|/D$ many nodes). In
this case  $k$-LD-convergence of the graph sequence $\G_n$
means that the exponent of the number of colorings achieving
the
quotient $(x,X)$ is well defined and is given by $\log
k-I(x,X)$. In other words, the total number of $k$-coloring
achieving the quotient $(x,X)$ is approximately
$\exp\left(|V_n|(\log k-I(x,X))\right)$. In  case $(x,X)$ is
not ''nearly'' achievable by any coloring, we have
$I(x,X)=\infty$, meaning that there are
 simply no coloring
achieving the target quotient $(x,X)$.

While
Definition~\ref{definition:GraphConvergenceLDk} might be more
intuitive than
Definition~\ref{definition:GraphConvergenceLDGeneral}, the
latter is more powerful in the sense that it does not require
a
''for all $k$'' condition, namely it is introduced through
having the LDP on \emph{one} as opposed to infinite (for every
$k$) sequence of probability spaces. It is also more powerful,
in the sense that the measure $(\rho_n,\sigma_n)$ completely
defines the underlying graph, whereas this is not the case for
the $k$-quotients. Since the space $\Dk$ being a subset of a
Euclidian space is a much ''smaller'' than the space of
measures $\M^d$, it is  not too surprising that the
LD-convergence implies the $k$-LD-convergence. But it is perhaps
surprising  that the converse is true.

\begin{theorem}\label{theorem:LDconvergenceEquivalence}
A graph sequence $\G_n$ is LD-convergent if and only if it is
$k$-LD-convergent for every $k$.
\end{theorem}
\begin{remark}\label{cor:Ik-from-I}
Given the rate function $I$ for LD-convergence, the
rate function
$I_k:\Dk\to \R_+\cup \{\infty\}$ for $k$-LD-convergence is easy to calculate, and will
be given by
\[
I_k(x,X)=\inf_{(\rho,\mu)} I(\mu,\rho)
\]
where the $\inf$ goes over all
$(\rho,\mu)\in \M^1\times\M^2$ such that
\begin{equation}
\label{phi-k-def}
\begin{aligned}
x_{i+1}&=\rho\left(\Delta_{i,k}\right), \quad& 0\le i\le k-1;\\
X_{i+1,j+1}&=\mu\left(\Delta_{i,k}\times \Delta_{j,k}\right), \quad& 0\le i,j\le k-1,
\end{aligned}
\end{equation}
with $\Delta_{i,k}$ denoting the open interval  $\Delta_{i,k}=\left({i\over k},{(i+1)\over k}\right)$.

We will also give an explicit construction for the rate function $I$ from the rate functions $I_k$, see Corollary~\ref{cor:Ikhat} below.
\end{remark}

Before proving the theorem, we reformulate the notion of
$k$-LD-convergence in terms of random measures with piecewise
constant densities. We need some notation.  Given a positive
integer $k$, let $O_k^d\subset [0,1]^d$ be the open set of all
points such that no coordinate is of the form
$0,1/k,2/k,\ldots,1$. Let $\M_k^d\subset \M^d$ be the set of
measures supported on $O_k^d$, and let $\N^d_k\subset \M^d_k$
be the set of measures in $\M^d_k$ which have constant density
on rectangles
\begin{align*}
\Delta_{i_1,k}\times \cdots \Delta_{i_d,k}, \qquad 0\le i_1,\ldots, i_d\le k-1.
\end{align*}
 We
define a projection operator $T_k:\M^1\times\M^2\to
\N_k^1\times\N_k^2$ by mapping a pair of measures
$(\rho,\mu)\in \M^1\times\M^2$ into
$(\rho_k,\mu_k)=T_k(\rho,\mu)$ defined as the pair of measures
in $\N_k^1\times\N_k^2$ for which
\begin{equation}
\label{rhok-muk}
\begin{aligned}
\rho_k(\Delta_{i,k})&=\rho(\Delta_{i,k})\quad\text{for}\quad 0\leq i\leq k-1
\\
\mu_k\left(\Delta_{i,k}\times \Delta_{j,k}\right)&=\mu\left(\Delta_{i,k}\times \Delta_{j,k}\right)
\quad\text{for}\quad 0\leq i,j,\leq k-1.
\end{aligned}
\end{equation}
Given a real valued coloring $\sigma:V_n\to [0,1]$  let
$\sigma_k=t_k(\sigma)$ be the coloring $\sigma_k:V_n\to [k]$
defined by $\sigma_k(u)=\lceil k\sigma(k)\rceil$ if
$\sigma(u)>0$ and $\sigma_k(u)=1$ if $\sigma(u)=0$.  Note that
$t_k(\sigma_n)$ is an i.i.d.\ uniformly random coloring with
colors in $[k]$ if $\sigma_n$ is chosen to be an i.i.d.\
uniformly random coloring with colors in $[0,1]$.

\begin{lemma}\label{lemma:GraphConvergenceLDk}
Let $\G_n=(V_n,E_n)$, and let $\sigma_n:V_n\to [0,1]$ be an
i.i.d. uniformly random coloring of $V_n$. Then $\G_n$ is
$k$-LD-convergent if an only if
$T_k(\rho(\sigma_n),\mu(\sigma_n))$ obeys a LDP on
$\N^1_k\times \N^2_k$.
\end{lemma}
\begin{proof}
Let $\hat\sigma_n:V_n\to [k]$ be an i.i.d uniformly random
coloring of $V_n$ with colors in $[k]$, and define
$F_k:\Dk\rightarrow \N^1_k\times \N^2_k$ as follows: given
$(x,X)\in \Dk$ with $x=(x_i, 1\le i\le k)$ and $X=(X_{ij}, 1\le
i,j\le k)$, let $F_k((x,X)):=(\rho,\mu)\in \N_k^1\times\N_k^2$
be the piece-wise constant measures with densities $kx_{i}$ on
$\Delta_{i-1,k}$  and   $k^2X_{i,j}$ on $\Delta_{i-1,k}\times
\Delta_{j-1,k}$, respectively. Then $F_k$ is an invertible
continues function with continuous inverse. Thus, applying the
Contraction Principle,
$(\rho_k(\hat\sigma_n),\mu_k(\hat\sigma_n))=F_k((x(\hat\sigma_n),X(\hat\sigma_n)))$
satisfies the LDP on $\N^1_k\times \N^2_k$ if and only if
$(x(\hat\sigma_n),X(\hat\sigma_n))$ satisfies the LDP on $\Dk$.

To conclude the proof, we only need to show that
$T_k(\rho(\sigma_n),\mu(\sigma_n))$ has the same distribution
as $(\rho_k(\hat\sigma_n),\mu_k(\hat\sigma_n))$.  But this
follows immediately from the facts that  $t_k(\sigma_n)$ has
the same distribution as $\hat\sigma_n$ and that
$T_k(\rho(\sigma_n),\mu(\sigma_n))$ depends only on the total
mass in the intervals $\Delta_{k,i}$, so that
$T_k(\rho(\sigma_n),\mu(\sigma_n))=(\rho_k(t_k(\sigma_n)),\mu_k(t_k(\sigma_n)))$.
\end{proof}

\subsubsection{Proof of Theorem~\ref{theorem:LDconvergenceEquivalence}, Part 1: LD-convergence implies $k$-LD-convergence}

In view of the last lemma, it seems quite intuitive that
LD-convergence should imply $k$-LD-convergence, given that the
set of measures $\N^1_k\times \N^2_k$ is much smaller than the
set of measures $\M^1\times \M^2$. More formally, one might
hope to define a continuous map from $\M^1\times \M^2$ to
$\N^1_k\times \N^2_k$ (or to $\Dk$) and use the Contraction
Principle to prove the desired $k$-LD-convergence.

It turns out that we can't quite do that due to point masses on
points of the form $i/k, 1\le i\le k$. To address this issue,
we will first prove that we can restrict ourselves to
$\M_k^1\times\M_k^2$, and then define a suitable continuous
function from  $\M_k^1\times\M_k^2$ to $\Dk$ to apply the
Contraction Principle. We need the following lemma.

\begin{lemma}\label{lemma:PositiveMassMeasure}
Suppose $\G_n$ is LD-convergent with  rate function $I$.
Suppose $\rho$ is such that $\rho(\{x\})>0$ for some $x\in
[0,1]$. Then for every $\mu\in\M^2, I(\rho,\mu)=\infty$.
Similarly, suppose $\mu$ is  such that $\mu(\{x\}\times
[0,1])>0$ or $\mu([0,1]\times \{x\})>0$  for some $x\in [0,1]$.
Then for every $\rho\in\M^1, I(\rho,\mu)=\infty$.
\end{lemma}

\begin{proof}
Suppose $\rho$ is such that $\rho(\{x\})=\alpha>0$ for some
$x\in [0,1]$. Fix $\delta<\alpha/2$. Suppose $\rho'$ is such
that $d(\rho,\rho')\le \delta$. Then, by definition of the
Prokhorov metric,
\begin{align*}
\alpha=\rho(\{x\})\le \rho'(x-\delta,x+\delta)+\delta,
\end{align*}
implying $\rho'(x-\delta,x+\delta)\ge \alpha/2$. Thus the event
$(\rho(\sigma_n),\mu(\sigma_n))\in B((\rho,\mu),\delta)$
implies the event $\rho(\sigma_n)(x-\delta,x+\delta)\ge
\alpha/2$. We now estimate the probability of this event. For
this event to occur we need to have at least $\alpha/2$
fraction of values $\sigma_n(u), u\in V_n$ to fall into the
interval $(x-\delta,x+\delta)$. This occurs with probability at
most
\begin{align*}
\sum_{i\ge (\alpha/2)|V_n|} {|V_n| \choose i}(2\delta)^{i}\le 2^{|V_n|}(2\delta)^{(\alpha/2)|V_n|}.
\end{align*}
Therefore
\begin{align*}
\limsup_n |V_n|^{-1}\pr\left(\rho(\sigma_n)(x-\delta,x+\delta)\ge \alpha/2\right)\le \log 2+(\alpha/2)\log(2\delta),
\end{align*}
implying
\begin{align*}
\lim_{\delta\rightarrow 0}\limsup_n |V_n|^{-1}\pr\left((\rho(\sigma_n),\mu(\sigma_n))\in B((\rho,\mu),\delta) \right)=-\infty,
\end{align*}
and thus $I(\rho,\mu)=\infty$ utilizing property
(\ref{eq:Identity1}). The proof for the case $\mu(\{x\}\times
[0,1])>0$ or $\mu([0,1]\times \{x\})>0$ is similar.
\end{proof}

We now establish that LD-convergence implies $k$-LD-convergence.
To this end, we first observe that for an i.i.d. uniformly
random coloring $\sigma_n:V_n\to [0,1]$, we have that
$\rho(\sigma_n)\in \M^1_k$ and $\mu(\sigma_n)\in\M^2_k$ almost
surely (as the probability of hitting one of the points $i/k$ is
zero).

Next we claim that since $(\rho(\sigma_n),\mu(\sigma_n))$
satisfies the LDP in $\M^1\times \M^2$, it also does so in
$\M_k^1\times \M_k^2$ with the same rate function. For this we
will show that (\ref{eq:LDprinciple}) holds when closures and
interiors are taken with respect to the metric space
$\M_k^1\times \M_k^2$ as opposed to $\M^1\times \M^2$. Indeed,
fix any set $A\subset\M_k^1\times \M_k^2$. Its closure in
$\M_k^1\times \M_k^2$ is a subset of its closure in $\M^1\times
\M^2$ and the set-theoretic difference between the two consists
of measures $(\rho,\mu)$ such that $\rho$ assigns a positive
mass to some point $i/k$, or $\mu$ assigns a positive mass to
some segment $\{i/k\}\times [0,1]$ or segment $[0,1]\times
\{j/k\}$. By Lemma~\ref{lemma:PositiveMassMeasure} the large
deviations rate $I$ at these points $(\rho,\mu)$ is infinite,
so that the value $\inf I(\rho,\mu)$ over the closure of $A$ in
$\M_k^1\times \M_k^2$ and $\M^1\times \M^2$ is the same.
Similarly, the interior of every set $A\subset \M_k^1\times
\M_k^2$ is a superset of its interior in $\M^1\times \M^2$, and
again the set theoretic difference between the two interiors
consists of points with infinite rate $I$. This proves the
claim.

Consider $\phi_k:\M_k^1\times \M_k^2\rightarrow \Dk$, where
$\phi_k$ maps every pair of measures $\rho,\mu$ into total
measure assigned to interval $\Delta_{i,k}$ by $\rho$  and
assigned to rectangle $\Delta_{i,k}\times \Delta_{j,k}$ by
$\mu$, for $i,j=0,1,\ldots,k-1$. Namely
$\phi_k(\rho,\mu)=(x,X)$ where $(x,X)$ is given by
\eqref{phi-k-def}.

We claim that $\phi_k$ is continuous with respect to
the respective
metrics. Indeed, suppose $(\rho_n,\mu_n)\rightarrow (\rho,\mu)$
in $\M_k^1\times \M_k^2$. Since $\M^d$ is equipped with the
weak topology, then for every closed (open) set $A\subset
[0,1]$, we have $\limsup_n\rho_n(A)\le \rho(A)
~~(\liminf_n\rho_n(A)\ge \rho(A))$. The same applies to
$\mu_n,\mu$. Setting $A$ to be $[i/k,(i+1)/k]$ first and then
$(i/k,(i+1)/k)$, and using the fact that $\rho_n$ and $\rho$
are supported on $O_k^1$, namely $\rho_n(i/k)=\rho(i/k)=0$ for
all $i$, we obtain $\rho_n(i/k,(i+1)/k)\rightarrow
\rho(i/k,(i+1)/k)$. A similar argument applies to $\mu_n,\mu$.
This proves the claim.

Since $\phi_k$ is continuous,  by the Contraction
Principle~\cite{demzei98}, the image of
$(\rho(\sigma_n),\mu(\sigma_n))$ under $\phi_k$ satisfies the
LDP as well. But this means precisely that we have
$k$-LD-convergence with rate function $I_k$ as defined
in Remark~\ref{cor:Ik-from-I} (strictly speaking,
the Contraction Principle gives a rate function where the
$\inf$ is taken over all $(\rho,\mu)\in \M_k^1\times \M_k^2$,
but by Lemma~\ref{lemma:PositiveMassMeasure} we can extend
the $\inf$ to that larger set  $\M^1\times \M^2$
without changing the value of $I_k$).

\subsubsection{Proof of Theorem~\ref{theorem:LDconvergenceEquivalence}, Part 2: $k$-LD-convergence implies LD-convergence}
In order to prove that $k$-LD-convergence  for all $k$ implies
LD-convergence, we would like to use the rate functions
$I_k:\N_k^1\times\N_k^2\to \R_+\cup\{\infty\}$ associated with
the LDP for $T_k(\rho(\sigma_n),\mu(\sigma_n))$ from
Lemma~\ref{lemma:GraphConvergenceLDk} to construct a suitable
rate function $I:\M^1\times\M^2\to \R_+\cup\{\infty\}$ for the
random variables $(\rho(\sigma_n),\mu(\sigma_n))$.  The
existence of such a rate function will follow from a  suitable
monotonicity property, see Lemma~\ref{lemma:Ikkhat} and
Corollary~\ref{cor:Ikhat} below.

Before stating Lemma~\ref{lemma:Ikkhat}, we state and prove a
more technical lemma which we will use at several places in
this subsection. For two pairs of measures
$(\rho,\mu)),(\rho',\mu')\in\N_k^1\times \N_k^2$ we define
\[
d_{var}((\rho,\mu)),(\rho',\mu'))=
\max
\Bigl\{
\max_{A\subset [0,1]}\left|\rho(A)-\rho'(A)\right|,
\max_{B\subset [0,1]^2}
\left|\mu(B)-\mu'(B)\right|
\Bigr\}.
\]

\begin{lemma}
\label{lem:Tk-1/k-bound} \hskip 12pt

1) If $(\rho,\mu)\in \M^1_k\times \M^2_k$, then
\begin{equation}
\label{tk-1/k-bound}
d((\rho,\mu),T_k(\rho,\mu))\leq \frac 1k.
\end{equation}

2) If $(\rho,\mu),(\rho',\mu')\in\N_k^1\times \N_k^2$, then
\begin{equation}
\label{dvar-dProc}
d((\rho,\mu),(\rho',\mu'))
\leq
d_{var}((\rho,\mu)),(\rho',\mu'))\leq (4kD+1) d((\rho,\mu),(\rho',\mu')).
\end{equation}

3) If $(\rho,\mu),(\rho',\mu')\in\N_{2k}^1\times \N_{2k}^2$,
then
\begin{equation}
\label{dvar-Tk}
d_{var}(T_k(\rho,\mu)),T_k(\rho',\mu'))
\leq
d_{var}((\rho,\mu)),(\rho',\mu')).
\end{equation}

\end{lemma}

\begin{proof}

1) Set $(\rho_k,\mu_k)=T_k(\rho,\mu)$, and let $A\subset
[0,1]$. Define $A_k$ to be the set that is obtained by
replacing every non-empty set $A\cap \Delta_{i,k}$ with the
entire interval $\Delta_{i,k}$. Then $A\subset A_k\subset
A^{1\over k}$. We have $\rho(A\cap \Delta_{i,k})\le
\rho(\Delta_{i,k})=\rho_k(\Delta_{i,k})$ by the definition of
$\rho_k$. Let $\alpha=\{i: A\cap \Delta_{i,k}\ne\emptyset\}$.
Then
\[
\rho(A)=\sum_{i\in \alpha}\rho(A\cap \Delta_{i,k})
\le\sum_{i\in \alpha}\rho_k(\Delta_{i,k})
=\rho_k(A_k)
\le\rho_k(A^{1\over k}).
\]
Conversely,
\[
\rho_k(A)\le \rho_k(A_k)=\rho(A_k)\le \rho(A^{1\over k}).
\]
A similar argument is used for $\mu$.

2) The lower bound follows immediately from the definitions of
$d$ and $d_{var}$.   To prove the upper bound,
 we first show that
\[
\max_{A\subset [0,1]}\left|\rho(A)-\rho'(A)\right|\leq (2kD+1)d(\rho,\rho').
\]
To this end, we note the maximum is obtained when $A$ is of the
form $A=\bigcup_{i\in \alpha}\Delta_{i,k}$ for some
$\alpha\subset \{0,\dots,k-1\}$.  On the other hand, for sets
$A$ of this form, we have that
$\rho(A^\epsilon)\leq\rho(A)+2k\epsilon\rho(A^c)\leq
\rho(A)+2k\epsilon D$, a bound which follows from the fact that
$A^c=[0,1]\setminus A$ is a union of intervals of length at
least $1/k$.  Choosing $\epsilon=d(\rho,\rho')$ we then have
$\rho(A')\leq \rho(A^\epsilon)+\epsilon\leq \rho(A)
+(2kD+1)\epsilon$, which implies $\rho'(A)-\rho(A)\leq
(2kD+1)\epsilon=(2kD+1)d(\rho,\rho')$. Exchanging the roles of
$\rho$ and $\rho'$ gives $\rho(A)-\rho'(A)\leq
(2kD+1)d(\rho,\rho')$, which completes the claim.  In a similar
way, one proves that
 \[
\max_{B\subset [0,1]^2}\left|\mu(B)-\mu'(B)\right|\leq (4kD+1)d(\mu,\mu').
\]
The proof uses that for sets $B$ which are unions of sets of
the form $\Delta_{i,k}\times\Delta_{j,k}$ we have
$\mu(B^{\epsilon})\leq\mu(B)+4k\epsilon D$.

3) Let $(\rho_k,\mu_k)=T_k(\rho,\mu)$ and
$(\rho_k^\prime,\mu_k^\prime)=T_k(\rho',\mu')$. We need to show
that $d_{var}(\rho_k,\rho_k^\prime)\leq
d_{var}(\rho,\rho^\prime)$ and $d_{var}(\mu_k,\mu_k^\prime)\leq
d_{var}(\mu,\mu^\prime)$.  The first bound follows from the
definition of $T_k$ and the fact that the maximum in the
definition $ d_{var}(\rho_k,\rho_k^\prime)= \max_{A\subset
[0,1]}|\rho_k(A)-\rho_k^\prime(A)| $ is obtained when $A$ is of
the form $A=\bigcup_{i\in \alpha}\Delta_{i,k}$ for some
$\alpha\subset\{0,\dots,k-1\}$. Indeed, for such an $A$, one
easily shows that
$|\rho_k(A)-\rho_k^\prime(A)|=|\rho(A)-\rho^\prime(A)|$, which
in turn implies that $d_{var}(\rho_k,\rho_k^\prime)\leq
d_{var}(\rho,\rho^\prime)$.  The proof of the bound
$d_{var}(\mu_k,\mu_k^\prime)\leq d_{var}(\mu,\mu^\prime)$ is
similar.
\end{proof}

%

\begin{lemma}\label{lemma:Ikkhat}
Suppose $\G_n$ is $k$-LD-convergent, and let
$I_k:\N_k^1\times\N_k^2\to \R_+\cup\{\infty\}$ be
the rate function associated with
the LDP for $T_k(\rho(\sigma_n),\mu(\sigma_n))$.
Then
$I_{k}(T_k(\rho,\mu))\le I_{2k}(\rho,\mu)$ for every
$(\rho,\mu)\in\N^1_{2k}\times \N^2_{2k}$.
\end{lemma}

\begin{proof}
Set $(\rho_k,\mu_k)=T_k(\rho,\mu)$,  fix $\epsilon>0$, and set
$\epsilon'=(4kD+1)\epsilon$. By the bounds \eqref{dvar-dProc}
and \eqref{dvar-Tk}, we have that
\begin{equation}
\label{in-B-eps}
(\rho_k^\prime,\mu_k^\prime)=T_{k}(\rho',\mu')\in B(T_k(\rho,\mu),\epsilon')
\end{equation}
whenever $(\rho',\mu')\in B((\rho,\mu),\epsilon)$. On the other
hand, $T_k(\rho(\sigma_n),\mu(\sigma_n))=
T_k(T_{2k}(\rho(\sigma_n),\mu(\sigma_n)))$.  By
\eqref{in-B-eps}, this implies that
\begin{align*}
\pr\Big( T_k(\rho(\sigma_n),\mu(\sigma_n))\in B\Bigl(T_k(\rho,\mu),(4kD+1)\epsilon\Bigr)\Big)
\geq
\pr\Big( T_{2k}(\rho(\sigma_n),\mu(\sigma_n))\in B\Bigl((\rho,\mu),\epsilon\Bigr)\Big).
\end{align*}
Applying \eqref{eq:Identity1} we obtain the claim of the lemma.
\end{proof}

\begin{coro}
\label{cor:Ikhat} Suppose $\G_n$ is $k$-LD-convergent for all
$k$, let $I_k:\N_k^1\times\N_k^2\to \R_+\cup\{\infty\}$  be
the rate function associated with
the LDP for $T_k(\rho(\sigma_n),\mu(\sigma_n))$,
let $(\rho,\mu)\in
\M^1\times\M^2$, and let $B_k$ be the ball
$B((\rho,\mu),2k^{-1})\cap \N^1_{k}\times\N^2_{k}$. Then the
limit
\[
I(\rho,\mu)=\lim_{\ell\to\infty}\inf_{(\rho',\mu')\in B_{2^\ell}}I_{2^\ell}(\rho',\mu')
\]
exists.
\end{coro}

\begin{proof}
Let $k\in\Z_+$.  By Lemma~\ref{lem:Tk-1/k-bound}, we have that
$T_k(\rho',\mu')\in B_k$ whenever $(\rho',\mu')\in B_{2k}$.
Combined with Lemma~\ref{lemma:Ikkhat}, this proves that
\[
\inf_{(\rho',\mu')\in B_{2k}}I_{2k}(\rho',\mu')\geq
\inf_{(\rho',\mu')\in B_{2k}}I_{k}(T_k(\rho',\mu'))
\geq
\inf_{(\rho'',\mu'')\in B_{k}}I_{k}(T_k(\rho'',\mu'')).
\]
The claim now follows by monotonicity.
\end{proof}

We are now ready to prove that $k$-LD-convergence  for all $k$
implies LD-convergence. Let $\sigma_n:V_n\to [0,1]$  be an
i.i.d. uniformly random coloring of $V_n$. Recalling
Lemma~\ref{lemma:GraphConvergenceLDk}, we will assume  that
$T_k(\rho(\sigma_n),\mu(\sigma_n))$ obeys a LDP with some rate
function $I_k$. We will want to prove that
$(\rho(\sigma_n),\mu(\sigma_n))$ obeys a LDP with the rate
function $I$ defined in Corollary~\ref{cor:Ikhat}.


%

Fix $(\rho,\mu)\in \M^1\times \M^2$. By the
Lemma~\ref{lem:Tk-1/k-bound}, the fact that
$(\rho(\sigma_n),\mu(\sigma_n))\in\M_{k}^1\times \M^2_{k}$
almost surely, and the triangle inequality, we have that
\begin{align*}
\liminf_n |V_n|^{-1}&\log\pr\left((\rho(\sigma_n),\mu(\sigma_n))\in B((\rho,\mu),3k^{-1}\right)\\
&\ge \liminf_n |V_n|^{-1}\log\pr\left(T_{k}(\rho(\sigma_n),\mu(\sigma_n))\in B((\rho,\mu),2k^{-1})\right)\\
&\ge-\inf I_{k}(\rho',\mu'),
\end{align*}
where the infimum is taken over $(\rho',\mu')\in
B((\rho,\mu),2k^{-1})\cap \N^1_{k}\times \N^2_{k}$. Similarly,
\begin{align*}
\limsup_n |V_n|^{-1}&\log\pr\left((\rho(\sigma_n),\mu(\sigma_n))\in B((\rho,\mu),k^{-1})\right)\\
&\le \limsup_n |V_n|^{-1}\log\pr\left(T_{k}(\rho(\sigma_n),\mu(\sigma_n))\in B((\rho,\mu),2k^{-1})\right)\\
&\leq-\inf I_k(\rho',\mu'),
\end{align*}
where the infimum is again taken over $(\rho',\mu')\in
B((\rho,\mu),2k^{-1})\cap \N^1_{k}\times \N^2_{k}$.

 To
complete the proof we apply Proposition~\ref{prop:LDReverse} in
conjunction with Corollary~\ref{cor:Ikhat}, yielding that
$(\rho(\sigma_n),\mu(\sigma_n))$ obeys a LDP with rate function
$I$.

\subsection{Basic properties and examples}
We begin by showing that the definition of LD-convergence is
robust with respect to the equivalency relationship $\sim$ on
graphs.
\begin{theorem}\label{theorem:LDrobust}
If $\G_n$ is LD-convergent and $\tilde\G_n\sim\G_n$, then
$\tilde\G_n$ is also LD-convergent.
\end{theorem}

\begin{proof}
We apply Theorem~\ref{theorem:LDconvergenceEquivalence} and
show that if $\G_n$ is $k$-LD-convergent then $\tilde\G_n$ is
also $k$-LD-convergent. Fix any $k$, any $(x,X)\in\Dk$ and
$\epsilon>0$. We denote by $\tilde\pr$ the probability measure
associated with
 $( \tilde x(\sigma_n), \tilde X(\sigma_n))= \tilde\G_n/\sigma_n$
when $\sigma_n:V_n\rightarrow [k]$ is a random uniformly chosen
map on $V_n$. The definition of $\sim$ implies
that for all sufficiently large $n$
\begin{align*}
\tilde\pr\Big((\tilde x(\sigma_n),\tilde X(\sigma_n))\in B\left((x,X),\epsilon/2\right)\Big)\le
\pr\Big((x(\sigma_n),X(\sigma_n))\in B\left((x,X),\epsilon\right)\Big),
\end{align*}
implying
\begin{align*}
\lim_{\epsilon\rightarrow 0}\limsup_n |V_n|^{-1}&\log\tilde\pr\Big((\tilde x(\sigma_n),\tilde X(\sigma_n))\in B\left((x,X),\epsilon/2\right)\Big) \\
&\le \lim_{\epsilon\rightarrow 0}\limsup_n |V_n|^{-1}\log\pr\Big((x(\sigma_n),X(\sigma_n))\in B\left((x,X),\epsilon\right)\Big).
\end{align*}
Similarly, we establish that
\begin{align*}
\lim_{\epsilon\rightarrow 0}\liminf_n |V_n|^{-1}&\log\tilde\pr\Big((\tilde x(\sigma_n),\tilde X(\sigma_n))\in B\left((x,X),\epsilon\right)\Big) \\
&\ge \lim_{\epsilon\rightarrow 0}\liminf_n |V_n|^{-1}\log\pr\Big((x(\sigma_n),X(\sigma_n))\in B\left((x,X),\epsilon/2\right)\Big).
\end{align*}
Since $\G_n$ is $k$-LD-convergent, then from
(\ref{eq:Identity1}) we have
\begin{align*}
\lim_{\epsilon\rightarrow 0}\limsup_n |V_n|^{-1}\log\tilde\pr(\cdot)=\lim_{\epsilon\rightarrow 0}\liminf_n |V_n|^{-1}\log\tilde\pr(\cdot).
\end{align*}
Thus the same identity applies to $\tilde\pr$. Applying
Proposition~\ref{prop:LDReverse} we conclude that $\tilde\pr$
is LD-convergent.
\end{proof}

Let us give some examples of LD-convergent graph sequences. As
our first example, consider a sequence of graphs which consists
of a
 disjoint union of copies of a fixed
graph.

\begin{ex}\label{theorem:UnionConverging}
Let $\G_0$ be a fixed graph and let $\G_n$ be a disjoint union
of $n$ copies of $\G_0$.  Then $\G_n$ is LD-convergent.
\end{ex}

\begin{proof}
We use Theorem~\ref{theorem:LDconvergenceEquivalence} and prove
that $\G_n$ is $k$-LD-convergent for every $k$. Fix $k$, let
$\Sigma_k$ be the space of all possible $k$-colorings of
$\G_0$, and let $M=k^{|V(\G_0)|}$ denote the size of $\Sigma_k$.
Every $k$-coloring $\sigma_n:V_n\rightarrow [k]$
of the nodes of $\G_n$
can then be encoded as
$\sigma_n=(\sigma_i^n)_{1\le i\le n}$, where $\sigma_i^n\in \Sigma_k$ is a $k$-coloring
of the $i$-th copy of $\G_0$ in $\G_n$. Consider the $M$-dimensional
simplex $S_M$, i.e., the set of vectors
$(z_1,\ldots,z_{M})$ such that $\sum_{i\le M}z_i=1,
z_i\ge 0$. For every $m=1,\ldots,M$ and
$\sigma_n=(\sigma_i^n)_{1\le i\le n}$, let $z_m(\sigma_n)$ be
the number of $\sigma_i^n$-s which are equal to the $m$-element
of $\Sigma_k$ divided by $n$, and let
$z(\sigma_n)=(z_m(\sigma_n))_{ 1\le m\le M}\in
S_{M}$. By Sanoff's Theorem, the sequence $z(\sigma_n),
n\ge 1$ satisfies the LDP with respect to the metric space
$S_{M}$. Now consider a natural mapping from
$S_{M}$ into $\Dk$, where each $z=(z_m)_{1\leq m\le M}$ is
mapped into $(x,X)\in \Dk$ as follows.
 For each element of $\Sigma_k$ encoded by
some $m\le M$ and each $i\le k$, let $x(i,m)$ be the
number of nodes of $\G_0$ colored $i$ according to $m$, divided
by $|V(\G_0)|$. In particular $\sum_ix(i,m)=1$. Similarly, let
$X(i,j,m)$ be $2$ times the number of edges in $\G_0$ with end node
colors $i$ and $j$ according to $m$, divided by $|V(\G_0)|$. In
particular $\sum_{i,j}X(i,j,m)=2|E(\G_0)|/|V(\G_0)|$. Setting
$x_i=\sum_m x(i,m)z_m, X_{i,j}=\sum_m X(i,j,m)z_m$, this
defines the mapping $S_{M}$ into $\Dk$. This mapping is
continuous. Observe that the composition $\sigma_n\rightarrow
z(\sigma_n)\rightarrow (x,X)$ is precisely the construction of
the factor graph $\F=\G_n/\sigma_n$ via (\ref{k-quotient}).
Applying the Contraction Principle, we obtain LD-convergence of
the sequence $\G_n$.
\end{proof}

Perhaps a more interesting example is the case of
subgraphs of the $d$-dimensional lattice $\Z^d$, more precisely
the case of graph sequences $L_{d,n}$ with vertex sets
$V_{d,n}=\{-n,-n+1,\dots, n\}^d$.
It is known that the free energy of statistical mechanics
systems on these graphs has a
limit~\cite{SimonLatticeGases},
\cite{GeorgyGibbsMeasure}. In
our terminology this means that the sequence $(L_{d,n})$ is
right-convergent (see Section~\ref{section:OtherConvergence}).
Here we show that this sequence is also LD-convergent.

\begin{ex}\label{theorem:Lattice}
Let $d\geq 1$, let $V_{d,n}=\{-n,-n+1,\dots, n\}^d$, let
$E_{n,d}$ be the set of pairs $\{x,y\}\subset V_{d,n}$ of
$\ell_1$ distance $1$, and let $L_{d,n}=(V_{d,n},E_{d,n})$.
Then the sequence $(L_{d,n})$ is LD-convergent.
\end{ex}

\begin{proof}
We again apply Theorem~\ref{theorem:LDconvergenceEquivalence}
and establish $k$-LD-convergence for every $k$. Fix $(x,X)\in
\Dk$ and let
\begin{align*}
\bar I_k(x)=-\lim_{\delta\rightarrow 0}\limsup_n {\log \pr\Bigl((x(\sigma_n),X(\sigma_n))\in B((x,X),\delta)\Bigr)\over |V_{d,n}|},
\end{align*}
where $\sigma_n:V_{d,n}\rightarrow [k]$ is chosen uniformly at
random. The limit $\lim_{\delta\rightarrow 0}$ exists by
monotonicity. Fix an arbitrary $\delta>0$. Then again by
monotonicity
\begin{align*}
\limsup_n {\log \pr\Bigl((x(\sigma_n),X(\sigma_n))\in B((x,X),\delta/2)\Bigr)\over |V_{d,n}|}\ge -\bar I_k(x,X),
\end{align*}
and we can find $n_0$ so that
\begin{align}\label{eq:barIkdelta}
{\log \pr\Big((x(\sigma_{n_0}),X(\sigma_{n_0}))\in B((x,X),\delta/2)\Bigr)\over |V_{d,n_0}|}\ge -\bar I_k(x,X)-\delta.
\end{align}
Consider arbitrary $n\ge n_0$, and set
$q=\lfloor\frac{2n+1}{2n_0+1}\rfloor$, $M=q^d$ and
$m=(2n+1)^d-M(2n_0+1)^d$. The set of vertices $V_{d,n}$ can
then be written as a disjoint union of $M$ copies of
$V_{d,n_0}$ and $m$ sets consisting of one node each. Let
$\Hg_1,\ldots,\Hg_M$ be the induced subgraphs on the copies of
$V_{d,n_0}$, and let $\tilde L_{d,n}=(\tilde V_{d,n},\tilde
E_{d,n})$ be the union of $\Hg_1,\ldots,\Hg_M$. Then
\[|\tilde V_{d,n}|=|V_{d,n}|-m =|V_{d,n}|(1-O(n_0/n))
\]
and
\[
|\tilde E_{d,n}|=
|E_{d,n}|-O(Mn_0^{d-1})- O(m)
=|E_{d,n}|-O(n^d n_0^{-1})-O(n^{d-1}n_0).
\]

%
A uniformly random coloring $\sigma_n:V_{d,n}\rightarrow [k]$
then induces  a uniformly random coloring $\tilde
\sigma_n:\tilde V_{d,n}\to [k]$. Due to the above bounds, the
corresponding quotients
$L_{d,n}/\sigma_n=(x(\sigma_n),X(\sigma_n))$ and $\tilde
L_{d,n}/\tilde \sigma_n=(\tilde x(\tilde\sigma_n),\tilde
X(\tilde\sigma_n))$ are close to each other.  More precisely,
%
\begin{align*}
d\Bigl((x(\sigma_n),X(\sigma_n))\,,\,(\tilde x(\tilde \sigma_n),\tilde X(\tilde\sigma_n))\Bigr)= O(n_0/n)+O(1/n_0).
\end{align*}
Then we can find $n_0$ large enough so that
(\ref{eq:barIkdelta}) still holds, such that for all large
enough $n$,
\begin{align*}
d\Bigl((x(\sigma_n),X(\sigma_n))\,,\,(\tilde x(\tilde \sigma_n),\tilde X(\tilde\sigma_n))\Bigr)\le \delta/2.
\end{align*}
Then
\begin{align*}
\pr\left((x(\sigma_n),X(\sigma_n)\in B((x,X),\delta))\right)
\ge \pr\left((\tilde x(\tilde \sigma_n),\tilde X(\tilde \sigma_n))\in B((x,X),\delta/2)\right).
\end{align*}
On the other hand, since $\tilde X_{i,j}(\tilde\sigma_n)$ are
counted over a disjoint union of graphs identical to
$L_{d,n_0}$, then if  $\tilde\sigma_n$ is such that within each
graph $\Hg_r$ we have $d((x,X), (\tilde x(\tilde
\sigma_n),\tilde X(\tilde\sigma_n)))\le \delta/2$, then the
same applies to the overall graph $\tilde L_{d,n}$. Namely,
\begin{align*}
\pr\left((\tilde x(\tilde\sigma_n),\tilde X(\tilde\sigma_n))\in B(x,\delta/2)\right)
\ge
\Big(\pr\left((x(\sigma_{n_0}),X(\sigma_{n_0}))\in B(x,X,\delta/2)\right)\Big)^{ M},
\end{align*}
where the second expectation is with respect to uniformly
random coloring of $L_{d,n_0}$. Since $M\le
|V_{d,n}|/|V_{d,n_0}|$, we  obtain
\begin{align*}
{1\over |V_{d,n}|}\log \pr\Bigl(x({\sigma_n}), X(\sigma_n))\in B(x,\delta)\Bigr)
&\ge {1\over |V_{d,n_0}|}\log\pr\Bigl((x(\sigma_{n_0}),X(\sigma_{n_0}))\in B(x,X,\delta/2)\Bigr)\\
&\ge -\bar I_k(x,X)-\delta,
\end{align*}
where the second inequality follows from (\ref{eq:barIkdelta}).
Since this holds for all large enough $n$, then from the bound
above we obtain
\begin{align*}
\lim_{\delta\rightarrow 0}\liminf_n{1\over |V_{d,n}|}\log \pr\Bigl((x(\sigma_n),X(\sigma_n))\in B((x,X),\delta)\Bigr)\ge -\bar I_k(x,X).
\end{align*}
We conclude that the relation (\ref{eq:Identity1}) holds.
\end{proof}

\section{Other notions of convergence and comparison with  LD-convergence}\label{section:OtherConvergence}
In this section we consider other types of convergence
discussed in the earlier literature and compare them with
LD-convergence. Later we will show that every other mode of
convergence, except for colored-neighborhood-convergence, is
implied by  LD-convergence. We begin with the notion of
left-convergence.

\subsection{Left-convergence}

We start with the definition of left-convergence as
introduced in \cite{BorgsChayesKahnLovasz}:
\begin{Defi}\label{defi:leftGraph-converging}
A sequence of graphs $\G_n$ with uniformly bounded degrees is
left-convergent if the limit
\begin{align}\label{eq:Left-limit}
\lim_{n\rightarrow\infty}|V_n|^{-1}\hom(\Hg,\G_n)
\end{align}
exists for every connected, finite graph $\Hg$.
\end{Defi}

As already noted in  \cite{BorgsChayesKahnLovasz}, this notion
is equivalent to the notion
 of
Benjamini-Schramm convergence of the sequence $\G_n=(V_n,E_n)$.
To define the latter,  consider a fixed positive integer $r$,
and a rooted graph $\Hg$ of diameter at most $r$.  Define
$p_n(\Hg,r)$ to be the probability that the $r$-ball around a
vertex $U$ chosen uniformly at random from $V_n$ is isomorphic
to $\Hg$. The sequence is called Benjamini-Schramm convergent
if the limit $p(\Hg,r)=\lim_n  p_n(\Hg,r)$ exists for all $r$
and all routed graphs $\Hg$ of diameter $r$ or less.


Recall our definition of relation $\sim$ between graph
sequences. It is immediate that if $\G_n$ is left-convergent
and $\tilde\G_n\sim\G_n$ then $\tilde\G_n$ is also
left-convergent with the same set of limits $p(\Hg,r)$. Thus
left-convergence can be defined on the equivalency classes of
graph sequences.

\subsection{Right-convergence}\label{subsection:right-convergence}
For sequences of bounded degree graphs, the notion of
right-convergence for a fixed target graph $\Hg$ was defined
in~\cite{BorgsChayesKahnLovasz}. In this paper, we define
right-convergence without referring to a particular target graph, and
require instead the existence of limits for all soft-core
graphs $\Hg$. Recall that for a weighted graph $\Hg$ with
vertex and edge weights given by the vector
$\alpha=(\alpha_1(\Hg),\dots,\alpha_k(\Hg))$ and the matrix
$A=(A_{ij}(\Hg))_{1\leq i,j\leq k}$, respectively,  the
homomorphism number $\hom(\G,\Hg)$ is given by
(\ref{eq:WeightedHom}). Without
loss of generality we assume $\alpha_i(\Hg)>0$ for all $i$. We
say that $\Hg$ is soft-core if also $A_{ij}(\Hg)>0$ for all
$i,j$.

To motivate our definition of right-convergence, let us note
that for a soft-core graph $\Hg$ on $k$ nodes,
\[
\alpha_{\min}^{(D+1)|V_n|}
\leq
\alpha_{\min}^{|V_n|+|E_n|}
\leq
k^{-|V_n|}\hom(\G_n,\Hg)
\leq
\alpha_{\max}^{|V_n|+|E_n|}
\leq
\alpha_{\max}^{(D+1)|V_n|}
\]
where
\begin{equation}\label{eq:alphamax}
\alpha_{\max}=\max\left(1,\max\alpha_i,\max A_{ij}\right)
\qquad\text{and}\qquad
\alpha_{\min}=\min\left(1,\min\alpha_i,\min A_{ij}\right).
\end{equation}
Thus  $\log \hom(\G_n,\Hg)$ grows linearly with $|V_n|$.  We
will define a sequence to  be right-convergent if the
coefficient of proportionality converges for all soft-core
graphs.

\begin{Defi}\label{defi:right-graphs converging}
A graph sequence $\G_n$ is defined to be right-convergent if
for every soft-core graph $\Hg$, the limit
\begin{align}\label{eq:LimitPartitionFunctin}
-f(\Hg)\triangleq \lim_{n\rightarrow\infty}{\log \hom(\G_n,\Hg)\over |V_n|},
\end{align}
exists.
\end{Defi}
\noindent In the language of statistical physics, the quantity
$\hom(\G,\Hg)$ is usually called the partition function of the
soft-core model with interaction $\Hg$ on $\G$, and the
quantity $f(\Hg)$ is called its free energy.
Sometimes we will write $f_{(\G_n)}(\Hg)$ instead of $f(\Hg)$
to emphasize the dependence on the underlying graph sequence
$G_n$.

 Observe that our definition
is robust with respect to the equivalency notion $\sim$ on
graph sequences. Namely if $\tilde\G_n\sim\G_n$, then the
limits (\ref{eq:LimitPartitionFunctin}) are the same.
Indeed,  deleting or adding one edge in $\G_n$ increases the
partition function  by a factor at most  $\max A_{ij}$ and
decreases it by a factor at most $\min A_{ij}$. Thus
adding/deleting $o(|V_n|)$ edges changes the partition
function of a soft-core model by a factor of $\exp(o(|V_n|)$.

By contrast, the partition function of a hard-core model
can be quite sensitive to adding or deleting edges, as we
already demonstrated in the introduction, where we discussed
the case of cycles. The same observation can be extended to
the case when $\G_n$ is a $d$-dimensional cylinder or torus,
again parity of $n$ determining the two-colorability property,
see~\cite{BorgsChayesKahnLovasz} for details. An even more
trivial example is the following: let $\G_n$ be a graph on $n$
isolated node (no edges) when $n$ is even, and $n-2$ isolated
nodes plus one edge when $n$ is odd. Then there exists
$1$-coloring of $\G_n$ if and only if $n$ is even, again
leading to the conclusion that $\G_n$ is not  converging on all
simple graphs $\Hg$.

%
%

We adapted Definition~\ref{defi:right-graphs converging}
to avoid these pathological examples.  Nevertheless,  we still
would like to be able to define the limits
(\ref{eq:LimitPartitionFunctin}) for hard-core graphs $\Hg$,
and do it in a way which is insensitive to changing from $\G_n$
to some equivalent sequence $\tilde\G_n\sim\G_n$. There are two
ways to achieve this which are equivalent, as we establish
below. Given $\lambda>0$ and  a (not necessarily
soft-core graph) $\Hg$ with edge weights given by a matrix
$A=(A_{ij}(\Hg))$, let
$A_\lambda$ be the  matrix with positive
entries
\[
(A_\lambda)_{ij}=\max\{\lambda,A_{ij}(\Hg)\}
\]
and let $\Hg_\lambda$ be the corresponding soft-core graph.

\begin{theorem}\label{theorem:RightConvergenceEquivalence}
Given a right-convergent sequence $\G_n$ and a weighted graph
$\Hg$, the following limit exists
\begin{align}\label{eq:hepsilon}
f_{(\G_n)}(\Hg)\triangleq \lim_{\lambda\rightarrow 0} f_{(\G_n)}(\Hg_\lambda).
\end{align}
Furthermore, there exists a graph sequence $\tilde\G_n\sim
\G_n$ such that for every graph $\Hg$
\begin{align}\label{eq:HardSoftEqual}
-f_{(\G_n)}(\Hg)=\lim_n{\log\hom(\tilde\G_n,\Hg)\over |V_n|}\ge
\sup_{\hat\G_n\sim\G_n}\limsup_n{\log\hom(\hat\G_n,\Hg)\over |V_n|},
\end{align}
where the supremum is over all graphs sequences $\hat\G_n$
which are equivalent to $\G_n$. In particular, the maximizing
graph sequence $\tilde\G_n$ exists  and can be chosen in such a
way that the $\limsup$ is actually a limit.
\end{theorem}
The intuition of the definition above is that we define
$f_{(\G_n)}(\Hg)$ by slightly softening the "non-edge"
requirement of zero elements of the weight matrix $A$. This
turns out to be equivalent to the possibility of removing
$o(|V_n|)$ edges in the underlying graph in trying to achieve
the largest possible limit within the equivalency classes of
graph sequences. One can further give a statistical physics
interpretation of this definition as defining the free energy at
zero temperature as a limit of free energies at positive
temperatures.

\begin{proof}
The existence of the limit (\ref{eq:hepsilon}) follows
immediately by monotonicity: note that $\hom(\G,\Hg)$ is
monotonically non-decreasing in every element of the weight
matrix $A$.

The proof of the second part is more involved, even
though the intuition behind it is again simple: First, one
easily shows that monotonicity in $\lambda$ and the fact that
the limit \eqref{eq:LimitPartitionFunctin} is not changed when
we change $o(|V_n|)$ edges implies that for every
$\tilde\G_n\sim\G_n$
\begin{align}\label{eq:limsuptildeGn}
\limsup_n{\log \hom(\tilde\G_n,\Hg)\over |V_n|}\le -f_{(\G_n)}(\Hg).
\end{align}
To prove that $f_{(\G_n)}(\Hg)$ is achieved asymptotically by
some $\tilde\G_n\sim \G_n$ we may assume without loss of
generality that $f_{(\G_n)}(\Hg)<\infty$, since otherwise the
identity holds trivially. Consider a ``typical configuration''
$\sigma:V_n\rightarrow V(\Hg_\lambda)$ contributing to
$\hom(\G_n,\Hg_\lambda)$, and let $E_0(\sigma)$ be the set of
edges $\{u,v\}\in E_n$ such that
$A_{\sigma(u),\sigma(v)}(\Hg)=0$,
\begin{equation}
\label{E_0-def}
E_{0}(\sigma)=\{(u,v)\in E_n: A_{\sigma(u),\sigma(v)}(\Hg)=0\}.
\end{equation}
For small $\lambda$, we expect the size of this set to grow
slowly with $|V_n|$, since otherwise $\hom(\G_n,\Hg_\lambda)$
would become too small to be consistent with $\lim_{\lambda\to
0}f_{(\G_n)}(\Hg_\lambda)<\infty$. We might therefore hope that
we can find a subset $E_{0,n}\subset E_n$ such that (i)
$|E_{0,n}|=o(|V_n|)$ and (ii) removing the set of edges
$E_{0,n}$ from $E_n$ leads to a graph sequence $\tilde G_n$
such that $|V_n|^{-1}\log \hom(\tilde\G_n,\Hg)$ is close to
$-f_{(\G_n)}(\Hg)$.

It will require a little bit of work to make this rather vague
argument precise. Before doing so, let us give the
 proof of \eqref{eq:limsuptildeGn}, which is much simpler:
Indeed, by monotonicity,  for for every $\lambda>0$ we have
\begin{align*}
\limsup_n{\log \hom(\tilde\G_n,\Hg)\over |V_n|}&\le \limsup_n{\hom(\tilde\G_n,\Hg_\lambda)\over |V_n|} \\
&=\limsup_n{\log \hom(\G_n,\Hg_\lambda)\over |V_n|}\\
&=\lim_n{\log \hom(\G_n,\Hg_\lambda)\over |V_n|} \\
&=-f_{(\G_n)}(\Hg_\lambda),
\end{align*}
where the first equality follows since $\Hg_\lambda$ is a
soft-core graph. Passing to the limit $\lambda\rightarrow
0$, we obtain \eqref{eq:limsuptildeGn}, which in turn
implies
\begin{align*}
\sup_{\tilde\G_n\sim\G_n}\limsup_n{\log \hom(\tilde\G_n,\Hg)\over |V_n|}\le -f_{(\G_n)}(\Hg).
\end{align*}

We now show that $f_{(\G_n)}(\Hg)$ is achieved asymptotically
by some $\tilde\G_n\sim \G_n$.
Our main technical result leading to the claim is as follows.

\begin{lemma}\label{lemma:Gnhatepsilon}
For every $\epsilon\in(0,1)$, every graph $\Hg$, and
for all large enough $n$, there exists a graph $\tilde \G_n$
which is obtained from $\G_n$ by deleting at most
$\epsilon|V_n|$ edges such that
\begin{align}\label{eq:Hgplusepsilon}
\hom(\tilde \G_n,\Hg)&\ge
\exp\left(-(f_{(\G_n)}(\Hg)+\epsilon)|V_n|\right).
\end{align}
\end{lemma}

We first show that this lemma implies the required claim. It
relies on a standard diagonalization argument. Note that it
suffices to establish the result for graphs $\Hg$ with rational
(possibly zero) weights. Let $\Hg_1,\Hg_2,\ldots$ be an
arbitrary enumeration of graphs with rational weights. For
every $m=1,2,\ldots~$, we find large enough $n(m)$ such that
for $n\ge n(m)$,
the claim in lemma holds for $\epsilon=1/m^2$ and for all
$\Hg=\Hg_1,\ldots,\Hg_m$. Namely, for every
$n\ge n(m)$, graphs $\tilde\G_{n,1},\ldots,\tilde\G_{n,m}$ can
be obtained from the graph $\G_n$ each be deleting at most
$n/m^2$ of edges in $\G_n$ such that (\ref{eq:Hgplusepsilon})
holds for $\epsilon=1/m^2$. Without  loss of generality we
may assume that $n(m)$ is strictly increasing in $m$. Let
$\tilde E_{n,i}$ be the set of edges deleted from $\G_n$ to
obtain $\G_{n,i}$ for $1\le i\le m$ and $\tilde\G_n^m$ be the
graph obtained from $\G_n$ by deleting the edges
in $\cup_{1\le i\le
m}E_{n,i}$. The number of deleted edges is at most
$m(n/m^2)=n/m$. By monotonicity, (\ref{eq:Hgplusepsilon})
implies that
\begin{align*}
\hom(\tilde \G_n^m,\Hg_i)&\ge
\exp\left(-(f_{(\G_n)}(\Hg_i)+1/m^2)|V_n|\right), \qquad i=1,2,\ldots,m.
\end{align*}
We now construct the graph sequence $\tilde G_n$. For all
$n<n(1)$ we simply set $\tilde G_n=\G_n$. For all other $n$, we
find a unique $m_n$ such that $n(m_n)\le n<n(m_n+1)$ and set
$\tilde\G_n=\tilde\G_n^{m_n}$. Then $\tilde\G_n\sim\G_n$ since
the number of deleted edges is at most
$|V_n|/m_n=o(|V_n|)$  as $n\rightarrow\infty$. For this
sequence we have
\begin{align*}
|V_n|^{-1}\log\hom(\tilde \G_n,\Hg_i)&\ge
-f_{(\G_n)}(\Hg_i)-1/m_n^2, \qquad i=1,2,\ldots,m_n.
\end{align*}
Fixing $i$ and taking $\liminf_n$ of each side we obtain for
each $i$
\begin{align*}
\liminf_n|V_n|^{-1}\log\hom(\tilde \G_n,\Hg_i)&\ge
-f_{(\G_n)}(\Hg_i).
\end{align*}
Combining with (\ref{eq:limsuptildeGn}) we conclude that for
each $i$
\begin{align*}
\lim_n|V_n|^{-1}\log\hom(\tilde \G_n,\Hg_i)&=
-f_{(\G_n)}(\Hg_i).
\end{align*}
This completes the proof of the theorem.
\end{proof}

\begin{proof}[Proof of Lemma~\ref{lemma:Gnhatepsilon}]
For the ease of exposition we write $f(\Hg)$ for
$f_{(\G_n)}(\Hg)$. If $f(\Hg)=\infty$, there is nothing to
prove, so we assume $f(\Hg)<\infty$.

Let
\begin{align*}
\alpha_\lambda(\sigma)\triangleq \prod_{u\in V_n}\alpha_{\sigma(u)}\prod_{(u,v)\in E_n}\max(A_{\sigma(u),\sigma(v)},\lambda)
\end{align*}
and let $\Sigma(E)$ be the set of $\sigma:V_n\to[k]$ such that
$E_0(\sigma)=E$.  Then
\begin{align*}
\hom(\G_n,\Hg_\lambda)=
\sum_{E\subset E_n}
\sum_{\sigma\in\Sigma(E)}\alpha_\lambda(\sigma).
\end{align*}
We will prove that for $\lambda$ sufficiently small and $n$
sufficiently large, we can find a set $E_{0}\subset E_n$ such
that $|E_0|\le \epsilon |V_n|$ and
\begin{equation}
\sum_{\sigma\in\Sigma(E_{0})}\alpha_\lambda(\sigma)\ge
\exp\left(-(f(\Hg)+\epsilon)|V_n|\right).
\label{eq:E0}
\end{equation}
Note that this bound immediately implies the claim of the
lemma.  Indeed, let $\tilde\G_n$ be the graph obtained from
$\G_n$ by deleting the edges in $E_{0}$ and assume with out
loss of generality that $\lambda$ is smaller than the smallest
non-zero entry of $A$. Then every $\sigma\in \Sigma(E_{0})$
satisfies
$\alpha_\lambda(\sigma)=\lambda^{|E_0|}\tilde\alpha(\sigma)$,
where $\tilde\alpha$ is the weight with respect to the graph
$\tilde \G_n$, vector $\alpha$ and the matrix $A$. This implies
\begin{align*}
\hom(\tilde \G_n,\Hg)&\ge \sum_{\sigma\in\Sigma(E_0)}\tilde\alpha(\sigma)
=\lambda^{-|E_0|}\sum_{\sigma\in\Sigma(E_0)}\alpha_\lambda(\sigma) \\
&\ge \sum_{\sigma\in\Sigma(E_0)}\alpha_\lambda(\sigma)
\geq \exp\left(-(f(\Hg)+\epsilon)|V_n|\right),
\end{align*}
as required. The proof of the lemma is therefore reduced to the
proof of \eqref{eq:E0}.

Given $0<\epsilon<1$ we chose  $0<\hat\lambda<\epsilon$  such
that
\begin{align}\label{eq:deltahat}
3\hat\lambda+\hat\lambda\log D+\hat\lambda\log(1/\hat\lambda)^{-1}<\epsilon,
\end{align}
and  $0<\lambda<\hat\lambda$  such that
\begin{align}\label{eq:deltaepsilon}
{f(\Hg)+2+(D+1)\log\alpha_{\max}+\log k\over \log(\lambda^{-1})}<\hat\lambda.
\end{align}
Here we note that $-f(\Hg)\le \log k+(D+1)\log\alpha_{\max}$, and thus the numerator in the
left-hand side is positive. Let $n_0=n_0(\lambda)$ be large
enough so that
\[
\log
(D|V_n|) \le \lambda|V_n|
\quad\text{and}\quad
\hom(\G_n,\Hg_\lambda)\ge \exp\left((-f(\Hg_\lambda)-\lambda)|V_n|\right)
\]
for all $n\geq n_0(\lambda)$.  Note that by monotonicity in
$\lambda$, the second bound implies that
\begin{align}
\hom(\G_n,\Hg_\lambda)&\ge \exp\left((-f(\Hg)-\lambda)|V_n|\right). \label{eq:n0}
\end{align}

 Fix an arbitrary such $n\geq n_0$, and let
$\Sigma(r)$ be the set of mappings $\sigma:V_n\rightarrow
V(\Hg)$ such that precisely $r$ edges of $V_n$ are mapped into
non-edges of $\Hg$,
\[
\Sigma(r)=\{\sigma:V_n\rightarrow
V(\Hg)\colon
\left| E_0(\sigma)\right|=r\}.
\]
For all $\lambda$ which are smaller than the smallest positive
element of $A$ we have
\begin{align*}
\hom(\G_n,\Hg_\lambda)=\sum_{0\le r\le |E_n|} \sum_{\sigma\in \Sigma(r)}\alpha_\lambda(\sigma).
\end{align*}
Let $r_0$ be such that
\begin{align}
\sum_{\sigma\in \Sigma(r_0)}\alpha_\lambda(\sigma)&\ge |E_n|^{-1}\hom(\G_n,\Hg_\lambda).
\label{eq:r0}
\end{align}
 We
claim that $ r_0\le \hat\lambda |V_n| $. Indeed, since for
every $\sigma\in\Sigma(r_0)$ we have $\alpha_\lambda(\sigma)\le
{ \lambda^{r_0}}\alpha_{\max}^{|V_n|+|E_n|} \le {
\lambda^{r_0}}\alpha_{\max}^{(D+1)|V_n|}$, it follows from
(\ref{eq:r0})  that
\begin{align*}
{k}^{|V_n|}\lambda^{r_0}\alpha_{\max}^{(D+1)|V_n|}
\ge |E_n|^{-1}\hom(\G_n,\Hg_\lambda)
\ge \exp\left((-f(\Hg)-2\lambda)|V_n|\right),
\end{align*}
where the second inequality follows from (\ref{eq:n0})  and the
fact that for $n\geq n_0$, we have  $|E_n|\leq D|V_n|\leq
e^{\lambda |V_n|}$.
 Rearranging, we get that
\begin{align*}
r_0\le (\log(1/\lambda))^{-1}\Big(|V_n|\big(f(\Hg)+{2}\lambda+{ (D+1)}\log\alpha_{\max}+\log { k}\big)\Big).
\end{align*}
Applying (\ref{eq:deltaepsilon}), this in turn implies that
\begin{align}
r_0&\le (\log\lambda^{-1})^{-1}|V_n|\left(f(\Hg)+2\lambda+(D+1)\log\alpha_{\max}+\log k\right)\notag \\
&\leq (\log\lambda^{-1})^{-1}|V_n|\left(f(\Hg)+2+(D+1)\log\alpha_{\max}+\log k\right)\notag\\
&\leq \hat\lambda |V_n|\label{eq:bound2 r_0},
\end{align}
as claimed.


Next we observe that
\begin{align*}
\sum_{\sigma\in \Sigma(r_0)}\alpha_\lambda(\sigma)=\sum_{E\subset E_n:\atop |E|=r_0}
\sum_{\sigma\in\Sigma(E)}\alpha_\lambda(\sigma).
\end{align*}
Since the number of subsets of the edges of $\G_n$ with
cardinality $r_0$ is  ${|E_n| \choose r_0}$, we can find a
subset $E_0\subset E_n, |E_0|=r_0$ such that
\begin{align}
\sum_{\sigma\in\Sigma(E_0)}\alpha_\lambda(\sigma)&\ge {|E_n| \choose r_0}^{-1}\sum_{\sigma\in \Sigma(r_0)}\alpha_\lambda(\sigma) \notag
\\
&\ge
 {|E_n| \choose r_0}^{-1}\exp\left((-f(\Hg)-2\lambda)|V_n|\right)\label{E0-bd}
\end{align}
where the second inequality follows from (\ref{eq:r0}).   Next,
we note that $ {N\choose k} \leq \Bigl(\frac{Ne}k\Bigr)^k, $ a
fact which can easily be proved by induction on $k$. Combining
this bound with the fact that $|E_n|\leq D|V_n|$ and the bounds
\eqref{eq:bound2 r_0} and \eqref{eq:deltahat}, we get
\begin{align*}
{|E_n| \choose r_0}\le {D|V_n| \choose r_0}\le {D|V_n| \choose {\hat\lambda} |V_n|}
\le \left(\frac {eD}{\hat\lambda}\right)^{\hat\lambda |V_n|}
= e^{((1+\log D)\hat\lambda +\hat\lambda\log(\hat\lambda^{-1}))|V_n|}
\leq e^{(\epsilon-2\hat\lambda)|V_n|}.
\end{align*}
Combined with \eqref{E0-bd} this implies the bound
\eqref{eq:E0}.
\end{proof}

\subsection{Partition-convergence and colored-neighborhood-convergence}
\label{subsection:partitionConvergence}

The notions of partition-convergence and
colored-neighborhood-convergence were introduced in Bollobas and
Riordan~\cite{BollobasRiordanMetrics}. Recall the notion of a
graph quotient $\F=\G/\sigma$ defined for every graph $\G$ and
mapping $\sigma:V(\G)\rightarrow [m]$. Then $\F$ is a weighted
graph on $m$ nodes with node and edge weights
$(x(\sigma),X(\sigma))\in \Dm$ defined by (\ref{k-quotient}).
When $\G$ is the $n$-th element of a graph sequence $\G_n$, we
will use notations $x_n(\sigma)$ and $X_n(\sigma)$.
Let $\Sigma_n(m)$ be the set of all pairs
$(x_n(\sigma),X_n(\sigma))$, when we vary over all maps
$\sigma$. As such $\Sigma_n(m)\in\mathcal{P}\left(\Dm\right)$,
where $\mathcal{P}(A)$  is the set of closed subsets of $A$.
For every set $A\subset\Dm$, let $A^\epsilon$ be the set of
points in $\Dm$ with distance at most $\epsilon$ to $A$ with
respect to the $\Lnorm_\infty$ norm on $\Dm$ (the actual choice
of metric is not relevant). Define a metric $\rho_m$ on closed
sets in $\mathcal{P}\left(\Dm\right)$ via
\begin{align}\label{eq:DugunjiMetric}
\rho(A,B)=\inf_{\epsilon>0}\{B\subset A^\epsilon,A\subset B^\epsilon\}.
\end{align}

\begin{Defi}
A graph sequence $\G_n$ is partition-convergent if for every
$m$, the sequence $\Sigma_n(m), n\ge 1$ is a Cauchy
sequence in the metric space $\mathcal{P}(\Dm)$.
\end{Defi}
It is known that $\mathcal{P}(\Dm)$ is a compact (thus closed)
metric space, which implies that  $\Sigma_n(m)$ has a
limit $\Sigma(m)\in \mathcal{P}(\Dm)$.
 Loosely speaking $\Sigma(m)$ describes the limiting set of
graph quotients achievable through the coloring of nodes by $m$
different colors. Thus, if for example some pair $(x,X)$ belong
to $\Sigma(m)$, this means that there is a sequence of
partitions of $\G_n$ into $m$ colors such that the normalized
number of nodes colored $i$ converges to $x_i$ for each $i$,
and the normalized number of  edges between colors $i$ and $j$
converges to $X_{i,j}$ for each $i,j$.

The definition above raises the question
whether  partition-convergence implies convergence of
neighborhoods, namely left-convergence.
We will show in the next section that this is not the case:
partition-convergence does not imply left-convergence (whereas,
as we mentioned above, right-convergence does imply the
left-convergence). In order to address this issue Bollobas and
Riordan extended their definition
of partition-convergence \cite{BollobasRiordanMetrics}
to a
richer  notion called
\emph{colored-neighborhood-convergence}. This notion was
further studied by Hatami, Lov{\'a}sz and
Szegedi~\cite{HatamiLovaszSzegedy} under the name
\emph{local-global} convergence, to stress the fact that this
notion captures both global and local properties of the
sequence $\G_n$.  We prefer the original name, given that both
right and LD-convergence capture local and global properties as
well.

To define  colored-neighborhood-convergence, we fix positive
integers $m$ and $r$. Let $\mathcal{H}_{m,r}\subset\mathcal{H}$
denote the finite set of all rooted colored graphs with radius
at most $r$ and degree at most $\Delta$, colored using colors
$1,\ldots,m$. For every graph $\G$ with degree $\le\Delta$ and
every node $u\in V(\G)$,  every coloring
$\sigma:V(\G)\rightarrow [m]$ produces an element $\Hg\in
\mathcal{H}_{m,r}$, which is the colored $r$-neighborhood
$B_\G(u,r)$ of $u$ in $\G$. For each $\Hg\in
\mathcal{H}_{m,r}$, let $N(\G,m,r,\sigma,\Hg)$ be the number of
nodes $u$ such that there exists a color matching isomorphism
from $B_\G(u,r)$ to $\Hg$. Namely, it is the number of times
that $\Hg$ appears as a colored neighborhood $B_\G(u,r)$ when
we vary $u$. Let $\pi(\G,m,r,\sigma)$ be the vector
$\left(|V(\G)|^{-1}N(\G,m,r,\sigma,\Hg), \Hg\in
\mathcal{H}_{m,r}\right)$. Namely, it is the vector of
frequencies of observing graphs $\Hg$, as neighborhoods
$B_\G(u,r)$ when we vary $u$. Then for every $m$, we obtain a
finite and therefore closed set $\pi(\G,m,r)\subset
[0,1]^{|\mathcal{H}_{m,r}|}$
 when we vary over all $m$-colorings $\sigma$ in
vectors  $\pi(\G,m,r,\sigma)$.  As such,
$\pi(\G,m,r)\in\mathcal{P}([0,1]^{|\mathcal{H}_{m,r}|})$. We
consider $\mathcal{P}([0,1]^{|\mathcal{H}_{m,r}|})$ as a metric
space with metric given by (\ref{eq:DugunjiMetric}).

\begin{Defi}\label{Defi:Local-GlobalConvergence}
A graph sequence $\G_n$ is defined to be
colored-neighborhood-convergent if the sequence $\pi(\G_n,m,r)$ is convergent in
$\mathcal{P}([0,1]^{|\mathcal{H}_{m,r}|})$ for every $m,r$.
\end{Defi}

It is not too hard to see that colored-neighborhood-convergence
implies partition-convergence and left-convergence. As we show
in the next section, the converse does not hold for either
partition- or left-convergence.

\section{Relationship between different notions of convergence}\label{section:ConvergenceRelations}
In this section we explore the relations between different
notions of convergence. As we will see, most of the definitions
are non-equivalent and many of them are not comparable in
strength. Our findings are summarized in the following theorem.

\begin{theorem}\label{theorem:ConvergenceRelations}
The notions of left, right, partition, colored-neighborhood and
LD-convergence satisfy the relations described on
Figures~\ref{figure:ConvergenceRelations} and
\ref{figure:ConvergenceRelationsColorNeighb}, where the solid
arrow between $A$ and $B$ means type $A$ convergence implies
type $B$ convergence, and the striped arrow between $A$ and $B$
means type $A$ convergence does not imply  type $B$
convergence.
\end{theorem}

We have deliberately described the relations in two figures,
since at this stage we do not know how LD and
colored-neighborhood-convergence are related to each other. As
a step towards this it would be interesting to see whether
colored-neighborhood-convergence implies  right-convergence. We
will discuss this question along with some other open questions
in Section~\ref{sec:discussion}.


Note that it was already proved in~\cite{BorgsChayesKahnLovasz} that right-convergence implies
left-convergence, and that it is obvious that
colored-neighborhood-convergence implies partition and
left-con\-ver\-gence.  To prove the Theorem, it will therefore be
enough to show that LD-convergence implies both right and
partition-con\-ver\-gence,  that left-convergence does not imply
right-convergence, that partition-convergence does not imply
left-convergence, and that right-convergence does not imply
partition-convergence (all other arrows are implied).  We will
start with the negative implications, which are all proved by
giving counter-examples.

\begin{figure}
\begin{center}
\scalebox{.6}{\ingraph{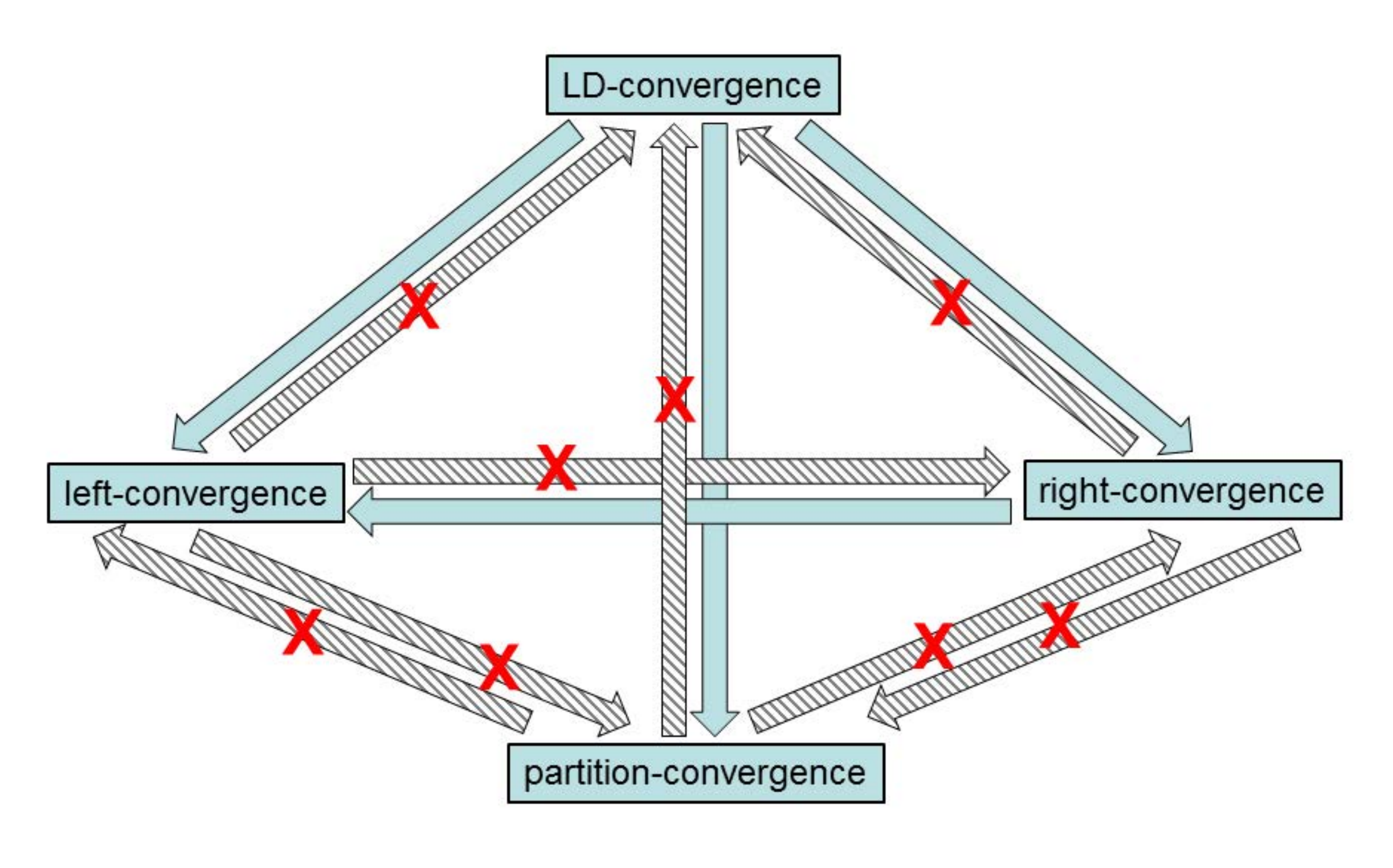}}
\caption{Implications between LD and other notions of convergence}\label{figure:ConvergenceRelations}
\end{center}
\end{figure}
\begin{figure}
\begin{center}
\scalebox{.6}{\ingraph{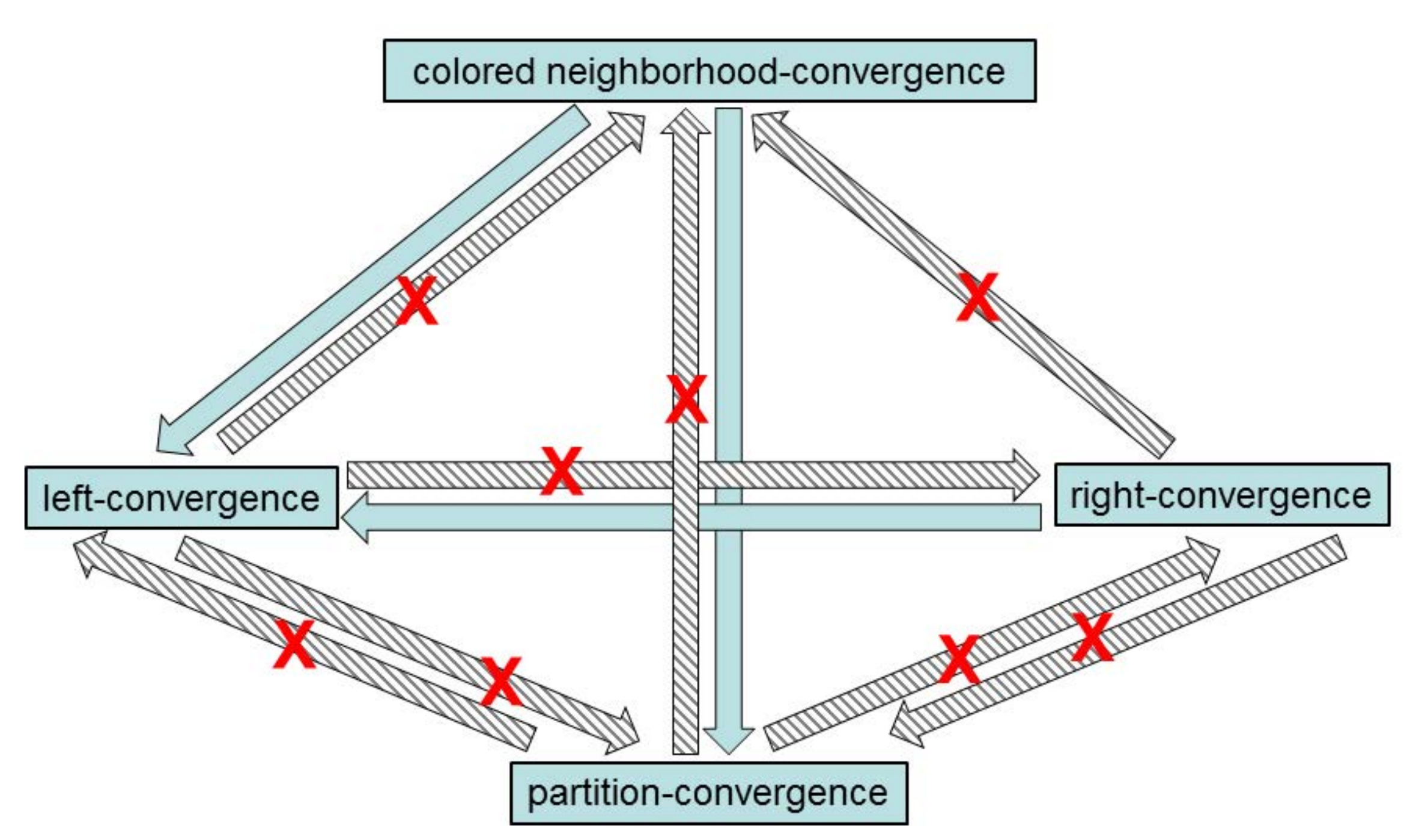}}
\caption{Implications between colored-neighborhood and other notions of convergence}\label{figure:ConvergenceRelationsColorNeighb}
\end{center}
\end{figure}

\subsection{Left-convergence does not imply right-convergence}
\label{sec:left-not-right}

As discussed earlier left-convergence does not imply
right-convergence on all hard-core graphs $\Hg$.  But our
modified definition of right-convergence involving only
soft-core graphs raises the question whether with our new
definition, left-convergence does imply right-convergence. It
turns out that this is not the case since right-convergence
implies  convergence of certain global properties like the
convergence of the maximal cut in a graph, a property which is
not captured by the notion of left-convergence.

Indeed, let $\G_n$ be right-convergent,  let
$\text{MaxCut}(\G_n)$  denotes the size of the maximal cut in
$\G_n$, and let $\Hg$ be a graph consisting of a single edge
with weights $\alpha=(1,1)$ and  $A_{1,2}=A_{2,1}=e^\beta,
A_{1,1}=A_{2,2}=1$.  Then
\[
e^{\beta\text{MaxCut}(\G_n)}
\leq \hom(\G_n,\Hg)
\leq 2^ne^{\beta\text{MaxCut}(\G_n)}.
\]
Sending $\beta\to\infty$, we see that right-convergence implies
convergence of $\frac 1{|V(\G_n)|}\text{MaxCut}(\G_n)$.

As an immediate consequence, we obtain that left-convergence
does not implies right-con\-ver\-gence.

\begin{ex}[\cite{BorgsChayesKahnLovasz}]
\label{Lem:left-not-right} Let $\Delta$ be a positive integer,
let $\G(n,\Delta)$  denote a random $\Delta$-regular graph
on $n$ nodes, and let $B_{n,\Delta}$ denote a random
$\Delta$-regular bi-partite graph on $n$ nodes.  Let $\G_n$ be
equal to $\G(n,\Delta)$ when $n$ is odd and equal to
$B_{n,\Delta}$  when $n$ is even.  Then $\G_n$ is
left-convergent for all $\Delta$, but for $\Delta$ sufficiently
large, it is not right-convergent.
\end{ex}

\begin{proof}
$\G_n$ is left-convergent since for every $r$, with high
probability $B(U,r)$ is a $\Delta$-Kelly (regular) tree
truncated at depth $r$, where $U$ is a random uniformly chosen
node in $\G_n$. We now show that $G_n$ is not right-convergent
for large $\Delta$. Indeed,
$\text{MaxCut}(\G_n)=\text{MaxCut}(B_{n,\Delta})=\frac 12\Delta
n $ for even $n$, while
$\text{MaxCut}(\G_n)=\text{MaxCut}(\G(n,\Delta))=\frac 14\Delta
n +O(\sqrt\Delta n)$ for odd $n$ \cite{BCP}.  This implies that
for large $\Delta$, the sequence $\G_n$ is not right-convergent.
\end{proof}

\subsection{Right-convergence does not imply partition-convergence}

In order to construct our counter example, we need the
following lemma. We recall that the edge expansion constant of
a sequence of graphs $\G_n$ is the largest constant $\gamma$
such that the number of edges between $W$ and $V(\G_n)\setminus
W$ is larger than $\gamma |W|$ whenever $|W|\leq |V(\G_n)|/2$.

\begin{lemma}
Let $\Delta\geq 3$.  Then there exists a $\gamma>0$ and a
right-convergent sequence of expanders $\G_n$ with edge
expansion constant $\gamma$ and maximal degree bounded by
$\Delta$.
\end{lemma}

\begin{proof}
Consider an arbitrary sequence $\G_n=(V_n,E_n)$ of
expanders with edge expansion $\gamma>0$ and maximal degree at
most $\Delta$, e.g., a sequence or 3-regular random graphs.
Construct a right-convergent subsequence of this sequence by
the following, standard diagonalization argument: enumerate the
countable collection of weighted graphs with rational positive
weights: $\Hg_1,\Hg_2,\ldots~$. For each $m=1,2,\ldots,$
construct a nested subsequences $n^{(m)}\subset n^{(m-1)}$,
such that $\G_n$ is right-convergent with respect
$\Hg_1,\ldots,\Hg_m$ along the subsequence $n^{(m)}$, and then
take the diagonal subsequence $n^{(m)}_m$ - the $m$-element of
the $m$-th sequence. The constructed sequence of graphs is
right-convergent with respect to every weighted graph $\Hg$
with rational coefficients. A straightforward argument shows
that then the sequence is right-convergent with respect to all
graphs with positive real weights, namely it is
right-convergent.
\end{proof}

\begin{ex} Given a right-convergent sequence of expanders $\G_n$
with bounded maximal degree and edge expansion $\gamma>0$, let
$\tilde \G_n=\G_m$ if $n=2m$ is even, and let $\tilde \G_n$ be a
disjoint union of two copies of $\G_m$ if $n=2m+1$ is odd.  Then
$(\tilde G_n)$ is right-convergent but not partition-convergent.
\end{ex}


\begin{proof}
Let $G_n=(V_n,E_n)$.
For every weighted graph $\Hg$,
\begin{align*}
|\tilde V_{2n+1}|^{-1}\log Z(\tilde \G_{2n+1},\Hg)&=(2|V_n|)^{-1}(2\log Z(\G_n,\Hg))\\
&=|V_n|^{-1}\log Z(\G_n,\Hg)\\
&=|\tilde V_{2n}|^{-1}\log Z(\tilde \G_{2n},\Hg).
\end{align*}
Since $\G_n$ is right-convergent,
$\tilde\G_n$ is right-convergent as well.

On the other hand, we claim that $\tilde\G_n$ is not
partition-convergent, since every cut of $\tilde \G_n=\G_m$
into two equal parts has size at least $\gamma |V_n|/2$ if $n$ is even, while
$\tilde\G_n$ has a cut into equal parts with zero edges when $n$ is odd.
Indeed,
denoting
the node set of $\tilde\G_n$ by $\tilde
V_n$,
consider the
set $\Sigma_n(2)$ of pairs
$(x_n(\sigma),X_n(\sigma))$ of vectors
$x_n(\sigma)=(x^n_1(\sigma),x^n_2(\sigma))$ and matrices
$X_n(\sigma,2)=(X^n_{i,j}(\sigma), 1\le i,j\le 2)$ when we
vary over $\sigma:\tilde V_n\rightarrow [2]$. Consider the projection
of $\Sigma_n(2)$ onto $[0,D]^3$ corresponding to
$(x^n_1(\sigma),x^n_2(\sigma),X^n_{1,2}(\sigma))$. Consider the
point $a=(1/2,1/2,0)\in [0,D]^3$ Observe that for each $n$,
there exists a $\sigma$ such that
$(x^{2n+1}_1(\sigma),x^{2n+1}_2(\sigma),X^{2n+1}_{1,2}(\sigma))=a$,
since we can split the graph $\tilde\G_{2n+1}$ into two parts with
the same number of nodes in each part and no edges in between
the parts. On the other hand for each graph $\tilde\G_{2n}$ and
each $\sigma$, the distance between
$(x^{2n}_1(\sigma),x^{2n}_2(\sigma),X^{2n}_{1,2}(\sigma))$ and
$a$ is bounded from below my $\min\{1/4,\gamma/4\}$ (indeed, if $|x^{2n}_1(\sigma)-1/2|<
1/4$, then by the expansion property of $G_n$,
$X^{2n}_{1,2}(\sigma) > \gamma/4$).
This shows that the sets
$\{(x^n_1(\sigma),x^n_2(\sigma),X^n_{1,2}(\sigma))\}$ obtained
by varying $\sigma$ are not converging as $n\rightarrow\infty$
and thus the sequence $\tilde\G_n$ is not partition-convergent.
\end{proof}

\subsection{Partition-convergence does not imply left-convergence}\label{subsection:PartitionDoesNotImplyLeft}
We now exhibit an example of a graph sequence which is
partition-convergent but not left-convergent.

\begin{ex}
Let $\G_n$ be a
disjoint union of $n$ $4$-cycles when $n$ is even and a
disjoint union of $n$ $6$-cycles when $n$ is odd. Then $\G_n$
is partition-convergent and not left-convergent.
\end{ex}

\begin{proof}
$\G_n$ is clearly not left-convergent. Let us show that it is
partition-convergent. For this purpose set $D=2$ and consider
arbitrary $k$. Let $\Sigma(k)$
be the set of all   $(x,X)\in \Dk$ such that
$\sum_{1\le i\le k} x_{i}=1$,
$X_{i,j}=X_{j,i}$ for all $i,j$, and
\begin{align}\label{eq:xijsum}
\sum_{j=1}^kX_{i,j}
=2x_{i}
\quad\text{for all }i.
\end{align}
We will prove that
$\Sigma_n(k)\rightarrow \Sigma(k)$ which shows that the sequence is
partition-convergent with $\Sigma(k)$ being the limit.

To this end, we first show that $\Sigma_n(k)\subset \Sigma(k)$
for all $n$.  Indeed, let $\sigma:V_n\to [k]$.
By the definition \eqref{k-quotient},
we have that
\[
 \sum_{j=1}^k X_{i,j}(\sigma)=
 \frac 1{|V_n|}\sum_{u\in V_i}\sum_{v\in V_n:\atop (u,v)\in E_n} 1
 = \frac 1{|V_n|}\sum_{u\in V_i} 2 =2x_i(\sigma),
\]
where we used that the degree of all vertices in $\G_n$ is two.  This proves that \eqref{eq:xijsum} holds for all
$(x,X)\in \Sigma_n(k)$.

We now claim that for every element $(x,X)\in \Sigma(k)$ and
every $n$ there exists $\sigma_n:V_n\rightarrow [k]$ such that
the $\Lnorm_\infty$ distance between $(x,X)$ and
$(x(\sigma_n),X(\sigma_n))$ is at most $k/n$.
This is clearly enough to
complete the proof, since it implies that for every
$\epsilon>0$ and $n$ large enough, we have  $\Sigma(k)\subset (\Sigma_n(k))^\epsilon$.

To prove the claim, we first consider even $n$ so that $\G_n$ is
a disjoint union of $n$  $4$-cycles.
To this end, let $n_{i,j}$, $i,j\in [k]$, be such that
\[
\Bigl|n_{i,j}-\frac n2 X_{i,j}\Bigr|\leq 1
\quad\text{for all }i,j\in [k]
\]
and
\[
\sum_{i,j\in [k]}n_{i,j}
=\sum_{i,j\in [k]}\frac n2 X_{i,j}=n.
\]
We consider the following
coloring $\sigma_n$ of $\G_n$: For each $i\ne j$
we take
$n_{i,j}$ cycles and color their $4$ vertices alternating between
color $i$ and $j$ as we go around these cycles, while for $i=j$,
we take $n_{i,i}$ cycles and color them with only one color, $i=j$.
For this coloring, we have
\[
X_{i,j}(\sigma)=X_{j,i}(\sigma)=
4\frac{n_{i,j}+n_{j,i}}{4n},
\]
so by our choice of $n_{i,j}$ we have that
\[
\Bigl|
X_{i,j}-X_{i,j}(\sigma)
\bigr|
\leq\frac 2n.
\]
Since $\Sigma_n(k)\subset \Sigma(k)$, we have
\[
|x_i-x_i(\sigma)|= \frac12\left|\sum_{j=1}^k\Bigl(X_{i,j}-X_{i,j}(\sigma)\Bigr)\right|
\leq \frac kn,
\]
completing the proof of the claim when $n$ is even.
 A
similar result is established for odd $n$, when $\G_n$ is a
disjoint union of $n$ $6$-cycles.
\end{proof}

\subsection{LD-convergence implies right-convergence}
\label{section:LD-implies-Right}

Let the sequence $\G_n$ be LD-convergent. Applying
Theorem~\ref{theorem:LDconvergenceEquivalence} it is also
$k$-LD-convergent with rate functions $I_k, k\ge 1$. Consider
an arbitrary $k$-node weighted graph $\Hg$ with strictly
positive node weights $\alpha=(\alpha_i, 1\le i\le k)$ and
strictly positive edge weights $A=(A_{i,j}, 1\le i,j\le k)$. We
define  an ``energy functional'' $\mathcal E_{\Hg}$ on $\Dk$
by
\begin{equation}
\label{E-def}
\mathcal E_{\Hg}(x,X)=
-\sum_{1\le i\le k}x_{i}\log \alpha_{i}
-\frac 12\sum_{1\le i,j\le k}X_{i,j}\log A_{i,j}.
\end{equation}
We also introduce an ``entropy functional'' $\mathcal S_k$ on
$\Dk$ by setting
\begin{equation}
\label{S-def}
\mathcal S_k(x,X)=\log k - I_k(x,X).
\end{equation}

\begin{theorem}
\label{thm:LD-implies-Right}
Suppose the sequence of graphs $\G_n$ is LD-convergent. Then it
is also right-convergent and
\begin{align}
\label{F-limit-sparse}
-\lim_{n\to\infty}{1\over |V_n|}\log Z(\G_n,\Hg)
=\inf_{(x,X)\in \Dk}\Bigl(\mathcal E_{\Hg}(x,X)-\mathcal S_k(x,X)\Bigr).
\end{align}
\end{theorem}

\noindent\begin{remark}
By the lower semi-continuity of $I_k$, the infimum over
$(x,X)\in\Dk$ is attained for some $(x^*,X^*)\in\Dk$.  The
right hand side can then be rewritten as $\mathcal
E_{\Hg}(x^*,X^*)-\mathcal S_k(x^*,X^*)$.  The theorem thus states
that the limiting free energy exists, and is given as the
energy minus entropy, an identity which is usually assumed
axiomatically in traditional thermodynamics.  Here it is a
consequence of LD-convergence.
\end{remark}

\noindent\begin{remark}
By Remark~\ref{cor:Ik-from-I}, we can express the limiting free energy directly in
terms of the rate function $I$.  Defining $\hat{\mathcal E}_{\Hg}(\rho,\mu)=\mathcal E_{\Hg}(x(\rho),X(\mu))$
where $x(\rho)$ and $X(\mu)$ are given by \eqref{phi-k-def}, this gives
\[
-\lim_{n\to\infty}{1\over |V_n|}\log Z(\G_n,\Hg)
=\inf_{(\rho,\mu)\in \M^1\times\M^2}\Bigl(\hat{\mathcal E}_{\Hg}(\rho,\mu)+I(\rho,\mu)-\log k\Bigr).
\]
\end{remark}

\begin{proof}[Proof of Theorem~\ref{thm:LD-implies-Right}]
Fix $\delta>0$. Recalling the definition \eqref{k-quotient},
let $\Sigma(x,X)$ be the set of colorings
$\sigma:V_n\rightarrow [k]$ such that
\begin{align*}
x_{i}-\delta \le
x_i(\sigma)
\le x_{i}+\delta.
\end{align*}
for every $i=1,\ldots,k$ and
\begin{align*}
X_{i,j}-\delta \le
X_{i,j}(\sigma)
\le X_{i,j}+\delta.
\end{align*}
for every $i,j=1,2,\ldots,k$. Note that we have the identity
\begin{align*}
|\Sigma(x,X)|=k^{|V_n|}\pr_{k,n}\Bigl(B((x,X),\delta)\Bigr),
\end{align*}
where we recall that $B((x,X),\delta)$ denotes
the closed $\delta$-ball around $(x,X)$ with respect to the
$\Lnorm_\infty$ norm, see
Subsection~\ref{subsection:DefProperties}.  Note also that
\[
\Bigl| \mathcal E_{\Hg}(x(\sigma),X(\sigma))-\mathcal E_{\Hg}(x,X)\Bigr|\leq K\delta,\qquad K=\Bigl(k+\frac{k^2}2\Bigr)\max\{\log\alpha_i,\log A_{i,j}\},
\]
whenever $\sigma\in\Sigma(x,X)$.

Define $\Gamma_\delta$ to be the set of pairs
  $(x,X)\in\Dk$ such that every
$x_i$ and every $X_{i,j}$ belongs to the set $\lbrace
0,\delta,2\delta,\ldots,\lceil D/\delta\rceil \delta\rbrace$.
Since $X_{i,j}(\sigma)\leq D$ for all $\sigma:V_n\to
[k]$, every $\sigma:V_n\to [k]$ lies in
$\cup_{(x,X)\in\Gamma_\delta}\Sigma(x,X)$. As a consequence,
\begin{align*}
Z(\G_n,\Hg)
&\le
\sum_{(x,X)\in\Gamma_\delta}\sum_{\sigma\in\Sigma(x,X)}
\prod_{u\in V_n}\alpha_{\sigma(u)}
\prod_{(u,v)\in  E_n}A_{\sigma(u),\sigma(v)}\\
&=
\sum_{(x,X)\in\Gamma_\delta}\sum_{\sigma\in\Sigma(x,X)}
e^{-\mathcal E_{\Hg}(x(\sigma),\,X(\sigma))|V_n|}
\\
&\le
\sum_{(x,X)\in\Gamma_\delta}k^{|V_n|}\pr_{k,n}\left(B((x,X),\delta)\right)
e^{-\mathcal E_{\Hg}(x,\,X)|V_n|+K\delta|V_n|}
\\
&\le
|\Gamma_\delta|\max_{(x,X)\in\Gamma_\delta}
\left(k^{|V_n|}\pr_{k,n}\Bigl(B((x,X),\delta)\Bigr)
e^{-\mathcal E_{\Hg}(x,\,X)|V_n|+K\delta|V_n|}\right).
\end{align*}
We obtain
\begin{align*}
{1\over |V_n|}\log Z(\G_n,\Hg)
&\le
{1\over |V_n|}\log|\Gamma_\delta|+\log k
\\
&+
\max_{(x,X)\in\Gamma_\delta}
\left({1\over |V_n|}\log \pr_{k,n}\Bigl(B((x,X),\delta)\Bigr)
-\mathcal E_{\Hg}(x,\,X)+K\delta\right),
\end{align*}
and hence
\begin{align*}
\limsup_n&{1\over |V_n|}\log Z(\G_n,\Hg)-\log k
\\
&\le
\limsup_n\max_{(x,X)\in\Gamma_\delta}
\left({1\over |V_n|}\log \pr_{k,n}\Bigl(B((x,X),\delta)\Bigr)
-\mathcal E_{\Hg}(x,\,X)\right)+K\delta
\\
&=
\max_{(x,X)\in\Gamma_\delta}
\left(\limsup_n{1\over |V_n|}\log \pr_{k,n}\Bigl(B((x,X),\delta)\Bigr)
-\mathcal E_{\Hg}(x,\,X)\right)+K\delta
\\
&\le
\max_{(x,X)\in\Gamma_\delta}
\left(-\inf_{(y,Y)\in B((x,X),\delta)}I_k(y,Y)
+\mathcal E_{\Hg}(x,\,X)\right)+K\delta
\\
&\le \max_{(x,X)\in\Gamma_\delta}
\left(-\inf_{(y,Y)\in B((x,X),\delta)}
\Bigl(I_k(y,Y)+\mathcal E_{\Hg}(y,\,Y)\Bigr)\right)+2K\delta
\\
&\le \max_{(x,X)\in\Gamma_\delta}
\left(-\inf_{(y,Y)\in \Dk}\Bigl(I_k(y,Y)+\mathcal E_{\Hg}(y,\,Y)\Bigr)\right)+2K\delta
\\
&=
-\inf_{(y,Y)\in \Dk}
\Bigl(I_k(y,Y)+\mathcal E_{\Hg}(y,\,Y)\Bigr)+2K\delta.
\end{align*}
Since this inequality holds for arbitrary $\delta$, we obtain
\begin{align*}
\limsup_n{1\over |V_n|}\log Z(\G_n,\Hg)\leq -\inf_{(x,X)\in \Dk}\Bigl(\mathcal E_{\Hg}(x,X)-\mathcal S_k(x,X)\Bigr).
\end{align*}
We now establish the matching lower bound. Fix
$\delta>0$ and let $(x^*,X^*)$ be the minimizer of $\mathcal
E_{\Hg}-\mathcal S_k$. We have
\begin{align*}
Z(\G_n,\Hg)&\ge\sum_{\sigma\in\Sigma(x^*,X^*)}
e^{-\mathcal E_{\Hg}(x(\sigma),\,X(\sigma))|V_n|}
\\
&\ge k^{|V_n|}\pr_{k,n}\Bigl(B((x^*,X^*),\delta)\Bigr)e^{-\mathcal E_{\Hg}(x^*,\,X^*)|V_n|-K\delta|V_n|},
\end{align*}
and thus
\begin{align*}
\liminf_n{1\over |V_n|}
&\log Z(\G_n,\Hg)-\log k\geq
\\
&\ge
\liminf_n{1\over |V_n|}\log \pr_{k,n}\Bigl(B((x^*,X^*),\delta)\Bigr)
-\mathcal E_{\Hg}(x^*,\,X^*)-K\delta
\\
&\ge
\liminf_n{1\over |V_n|}\log \pr_{k,n}\Bigl(B^o((x^*,X^*),\delta)\Bigr)
-\mathcal E_{\Hg}(x^*,\,X^*)-K\delta
\\
&\ge
-\inf_{(y,Y)\in B^o((x^*,X^*),\delta)}I_k((y,Y))
-\mathcal E_{\Hg}(x^*,\,X^*)-K\delta
\\
&\ge
-I_k(x^*,X^*)
-\mathcal E_{\Hg}(x^*,\,X^*)-K\delta.
\end{align*}
Letting $\delta\rightarrow 0$ we obtain the result.
\end{proof}

\subsection{LD-convergence implies partition-convergence}
We prove this result by identifying the limiting sets
$\Sigma(k), k\ge 1$.
In order to do this, we recall that by
Theorem~\ref{theorem:LDconvergenceEquivalence} every LD-convergent sequence
is
k-LD-convergent for all $k$.  We denote the corresponding functions by $I_k$.

\begin{theorem}\label{theorem:PartitionIsLDInfty}
Let
$\G_n$ be a sequence of graphs which is LD-convergent.
Then $\G_n$
is partition-convergent. Moreover, for each
$k$,
the sets $\Sigma_n(k)$ converge to
\begin{align}\label{eq:SigmakInfinite}
\Sigma(k)=\{(x,X)\in\Dk: I_k(x,X)<\infty\}.
\end{align}
\end{theorem}
The intuition behind this result is as follows. $\Sigma(k)$ is
the set of ''approximately achievable'' partitions $(x,X)$. If
some partition is approximately realizable as factor graph
$\G_n/\sigma$ for some $\sigma:V_n\rightarrow [k]$, then there
are exponentially many such realizations, since changing a
small linear in $|V_n|$ number of assignments of $\sigma$ does
not change the factor graph significantly. Therefore, a uniform
random partition $\sigma$ has at least an inverse exponential
probability of creating a factor graph $\G_n/\sigma$ near
$(x,X)$ and, as a result, the LD rate associated with $(x,X)$
is positive.

On the other hand if $(x,X)$ is $\epsilon$-away from any factor
graph $\G_n/\sigma$, then the likelihood that a uniform random
partition $\sigma$ results in a factor graph close to $(x,X)$
is zero, implying that the LD rate associated with $(x,X)$ is
infinite. We now formalize this intuition.

\begin{proof}
We define $\Sigma(k)$ as in (\ref{eq:SigmakInfinite}) and prove
that it is the limit of $\Sigma_n(k)$ as $n\rightarrow\infty$
for each $k$. We assume the contrary. This means that either
there exist $k$ and $\epsilon_0>0$ such that for infinitely
many $n$, $\Sigma_n(k)$ is not a subset of
$(\Sigma(k))^{\epsilon_0}$, or there exist $k$ and
$\epsilon_0>0$ such that for infinitely many $n$, $\Sigma(k)$
is not a subset of $(\Sigma_n(k))^{\epsilon_0}$.

We begin with the first assumption. Then, for infinitely many
$n$ we can find $(x_n,X_n)\in\Sigma_n(k)$ such that $(x_n,X_n)$
has distance at least $\epsilon_0$ from $\Sigma(k)$. Denote
this subsequence by $(x_{n_r},X_{n_r})$ and find any limit
point $(\bar x,\bar X)$ of this sequence. For notational
simplicity we assume that in fact $(x_{n_r},X_{n_r})$ converges
to $(\bar x,\bar X)$ as $r\rightarrow\infty$ (as opposed to
taking a further subsequence of $n_r$).
Then the
distance between $(\bar x,\bar X)$ and $\Sigma(k)$ is at least
$\epsilon_0$ as well. We now show that in fact $I_k(\bar x,\bar
X)<\infty$, implying that in fact $(\bar x,\bar X)$ belongs to
$\Sigma(k)$, which is a contradiction.

To prove the claim we fix $\delta>0$. We have that
$\Lnorm_\infty$ distance between $(x_{n_r},X_{n_r})$ and $(\bar
x,\bar X)$ is at most $\delta$ for all large enough $r$. Since
$(x_{n_r},X_{n_r})\in\Sigma_{n_r}(k)$, then for all  $r$ there
exists some $\hat\sigma_{n_r}:V_{n_r}\rightarrow [k]$ such that
$(x_{n_r},X_{n_r})$ arises from the partition
$\hat\sigma_{n_r}$. Namely,
$x_{n_r}=x(\G_{n_r}/\hat\sigma_{n_r})$ and
$X_{n_r}=X(\G_{n_r}/\hat\sigma_{n_r})$.
 Since a randomly uniformly chosen
map $\sigma_{n_r}:V_{n_r}\rightarrow [k]$ coincides with
$\hat\sigma_{n_r}$ with probability exactly $k^{-n_r}$,  we
conclude that
\begin{align*}
\pr\left((x(\sigma_{n_r}),X(\sigma_{n_r}))\in B((\bar x,\bar X),\delta)\right)\ge k^{-n_r},
\end{align*}
for all sufficiently large $r$. This implies
\begin{align*}
\limsup_n \frac 1n&\log\pr\left((x(\sigma_{n}),X(\sigma_{n}))\in B((\bar x,\bar X),\delta)\right)\\
&\ge\limsup_r \frac 1{n_r}\log\pr\left((x(\sigma_{n_r}),X(\sigma_{n_r}))\in B((\bar x,\bar X),\delta)\right)\\
&\ge -\log k.
\end{align*}
Applying property (\ref{eq:Identity1}) we conclude that
$I_k(\bar x,\bar X)\le \log k<\infty$, and the claim is
established.

Now assume that $k$ and $\epsilon_0>0$ are such that for
infinitely many $n$, $\Sigma(k)$ is not a subset of
$(\Sigma_n(k))^{\epsilon_0}$. Thus there is a subsequence of
points $(x_{n_r},X_{n_r})\in \Sigma(k)$ which are at least
$\epsilon_0$ away from $\Sigma_{n_r}(k)$ for each $r$. Let
$(\bar x,\bar X)\in \Dk$ be any limit point of
$(x_{n_r},X_{n_r}), r\ge 1$. By lower semi-continuity we have
that the set $\Sigma(k)$ is closed, and thus compact.
Therefore, we also have $(\bar x,\bar X)\in\Sigma(k)$. Again,
without  loss of generality, we assume that in fact
$(x_{n_r},X_{n_r})$ converges to $(\bar x,\bar X)$. Fix any
$\delta<\epsilon_0/2$. We have $(x_{n_r},X_{n_r})\in B((\bar
x,\bar X),\delta)$ for all sufficiently large $r$. Then for all
sufficiently large $r$ we have that the distance from $(\bar
x,\bar X)$ to $\Sigma_{n_r}(k)$ is strictly larger than
$\delta$ (since otherwise the distance from $(x_{n_r},X_{n_r})$
to this set is less than $\epsilon_0$). This means that for all
sufficiently large $r$,
\begin{align*}
\pr\left((x(\sigma_{n_r}),X(\sigma_{n_r}))\in B((\bar x,\bar X),\delta)\right)=0,
\end{align*}
implying
\begin{align*}
\liminf_n \frac 1n&\log\pr\left((x(\sigma_{n}),X(\sigma_{n}))\in B((\bar x,\bar X),\delta)\right)\\
&\le\liminf_r \frac 1{n_r}\log\pr\left((x(\sigma_{n_r}),X(\sigma_{n_r}))\in B((\bar x,\bar X),\delta)\right)\\
&=-\infty.
\end{align*}
Applying property (\ref{eq:Identity1}) we conclude that
$I_k(\bar x,\bar X)=\infty$, contradicting the fact that $(\bar
x,\bar X)\in \Sigma(k)$. This concludes the proof.
\end{proof}

\section{Discussion}
\label{sec:discussion}

\subsection{Right-convergence for dense vs. sparse graphs}

It is instructive to compare our results from Section~\ref{section:LD-implies-Right},
in particular the expression for the limiting free energy from Theorem~\ref{thm:LD-implies-Right},
to the corresponding results from the theory of convergent sequences for dense graphs.

For sequences $\G_n$ of dense graphs, i.e., graphs with average degree proportional to the number of vertices,
the  weighted homomorphism numbers $\hom(\G_n,\Hg)$  typically either grow or decay exponentially in $|V(\G_n)|^2$, implying
that the limit of an expression of the form \eqref{eq:rightlimit}
would be either $\infty$ or $-\infty$.  To avoid this problem, we adopt the usual definition from statistical physics,
where the free energy is defined by
\begin{align}
\label{F-def-dense}
f(\G_n,\Hg)=-\frac 1{|V(\G_n)|}\log Z(\G_n,\Hg),
\end{align}
where $Z(\G_n,\Hg)$ is the {\it partition function}%
\footnote{Note that the only difference with respect to \eqref{eq:WeightedHom} is the exponent $1/|V(\G_n)|$, which ensures that
$\log Z(\G_n,\Hg)$ is of order $O(|V_n|)$.}
\[
Z(\G_n,\Hg)=\sum_{\sigma:V(\G_n)\rightarrow V(\Hg)}\prod_{u\in V(\G_n)}\alpha_{\sigma(u)}\prod_{(u,v)\in E(\G_n)}\left(A_{\sigma(u),\sigma(v)}\right)^{1/|V(\G_n)|},
\]
with $\alpha$ and $A$ denoting the vector and matrix of the node and edges weights of $\Hg$, respectively.

It was shown in \cite{BorgsChayesEtAlGraphLimitsII} that for a left-convergent sequence of dense graphs,
the limit of \eqref{F-def-dense} exists, and can be expressed in terms of the limiting graphon.
We refer the reader to the literature \cite{LovaszSzegedy} for the definition
of the limiting graphon; for us, it is only important that it
is a measurable, symmetric function $W:[0,1]^2\to [0,1]$ which can be thought of as the limit
of a step function representing the adjacency matrices of $\G_n$ (see
Corollary 3.9 of  \cite{BorgsChayesEtAlGraphLimitsI} for a precise version of this statement).

To express the limit of the free energies \eqref{F-def-dense} in terms of  $W$,
we need some notation.
Assume $|V(\Hg)|=k$, and let $\Omega_k$ be the set of measurable functions
$\rho=(\rho_1,\dots,\rho_k):[0,1]\to [0,1]^k$ such that $\rho_i$ is non-negative for all $i\in [k]$ and
$\sum_i\rho_i(x)=1$ for all $x\in [0,1]$.  Setting
$
x_i(\rho)=\int_0^1 \rho_i(x)dx
$
and
$
X_{i,j}(\rho)=\int_0^1\int_0^1 \rho_i(x)\rho_j(y)W(x,y)dxdy,
$
we define the energy and entropy of a ``configuration'' $\rho\in \Omega_k$ as
\[
\mathcal E_{\Hg}(\rho) = -\sum_{1\leq i\leq k} x_i(\rho)\log\alpha_i
-\frac 12\sum_{1\le i,j\le k}X_{i,j}(\rho)
\log A_{i,j}
\]
and
\[
\mathcal S_k(\rho)=-\int_0^1 dx\sum_{1\leq i\leq k}\rho_i(x)\log\rho_i(x),
\]
respectively.  Then
\[
\lim_{n\to\infty}f(\G_n,\Hg)
=\inf_{\rho\in\Omega_k}\Bigl(\mathcal E_{\Hg}(\rho) - \mathcal S_k(\rho)\Bigr).
\]

Note the striking similarity to the expressions \eqref{E-def}, \eqref{S-def} and \eqref{F-limit-sparse}.  There is, however,
one important difference:
in the dense case, the limit object $W$ describes an effective edge-density and enters into the definition of the energy,
while in the sparse case, the limit object is a large-deviation rate and enters into the entropy.  By contrast,
the entropy in the dense case is trivial, and does not
depend on the sequence $\G_n$, similar to the energy in the sparse case, which does not depend on $\G_n$.

As we will see in the next subsection, this points to the fact that right-convergence
is a much richer concept for sparse graphs than for dense graphs.  Indeed, for
many interesting sequences of dense graphs, one can explicitly construct the limiting graphon $W$, reducing
free energy calculations to simple variational problems.  By contrast, we do not even know
how to calculate the rate functions for the simplest random sparse graph, the Erd\"os-Renyi random graph.
We discuss this, and how it is
related to spin glasses, in the next section.

\subsection{Conjectures and open questions}
While Theorem~\ref{theorem:ConvergenceRelations} establishes a
complete picture regarding the relationships between
LD, left, right and partition-convergence, several questions
remain open regarding  colored-neighborhood-convergence. In
particular, not only   do we not know whether LD implies
colored-neighborhood-convergence or vice verse, but we not even
know whether colored-neighborhood-convergence implies
right-convergence.

Since  colored-neighborhood-convergence
does not appear to take into account the counting measures
involved in the definition of  right-convergence, we
conjecture that in fact  colored-neighborhood-convergence
\emph{does not} imply  right-convergence. Should this be the
case it would also imply that
colored-neighborhood-convergence does not imply
LD-convergence, via Theorem~\ref{theorem:ConvergenceRelations}.
At the same time we conjecture that the LD-convergence implies
the colored-neighborhood-convergence. It is an
interesting
open problem to resolve these conjectures.

Perhaps the most important outstanding conjecture regarding the
theory of convergent sparse graph sequences concerns
convergence of random sparse graphs.
Consider two natural and much studied random
graphs sequences, the sequence of
Erd\"os-Renyi random graphs $\G(n,c/n)$, and the sequence of random
regular graphs $\G(n,\Delta)$.
As usual, $\G(n,c/n)$ is the random subgraph
of the complete graph on $n$ nodes where each edge is kept independently
with probability $c/n$, where $c$ is an $n$-independent constant, and
$\G(n,\Delta)$ is a
$\Delta$-regular graph on $n$ nodes, chosen uniformly
at random from the family of all $n$-node $\Delta$-regular
graphs. We ask whether these graph sequences are convergent
with respect to any of the definitions of graph convergence
discussed above%
\footnote{In contrast to $\G(n,\Delta)$, the graphs  $\G(n,c/n)$
do not have uniformly bounded degrees, which means that most of the theorems
proved in this paper do not hold.
Nevertheless, we can still pose the question of convergence of
these graphs with respect to various notions.}.

It is well known
and easy to show that these graphs sequences are
left-convergent, since the $r$-depth neighborhoods of a random
uniformly chosen point can be described in the limit as
$n\rightarrow\infty$ as the first $r$ generations of an appropriate
branching process. The validity for other notions of convergence,
however, presents a serious challenge as here we touch on the rich
 subject of mathematical aspects of the
spin glass theory. The reason is that  right-convergence of
sparse graph sequences is
nothing but
 convergence of normalized
partition functions of the associated Gibbs measures. Such
convergence means the existence of the associated thermodynamic
limit and this existence question is one of the outstanding
open areas in the theory of spin glasses.

Substantial progress has been achieved recently due to the powerful
interpolation method introduced by Guerra and
Toninelli~\cite{GuerraTon} in the context of the
Sherrington-Kirkpatrick model, and further by Franz and
Leone~\cite{FranzLeone}, Panchenko and
Talagrand~\cite{PanchenkoTalagrand}, Bayati, Gamarnik and
Tetali~\cite{BayatiGamarnikTetali}, Abbe and Montanari
\cite{AbbeMontanari} and
Gamarnik~\cite{GamarnikRightConvergence}.
 In particular, a broad class of weighted
graphs $\Hg$ is identified, with respect to which the random
graph sequence $\G(n,c/n)$ is right-convergent w.h.p. (in some
appropriate sense). Some extensions of these results are also
known for random regular graph sequences $\G(n,\Delta)$. A
large  general class of such weighted graphs $\Hg$ is described
in \cite{GamarnikRightConvergence} in terms of
suitable positive
semi-definiteness properties of the graph $\Hg$. Another class
of graphs $\Hg$ with respect to which convergence is known and
in fact the limit can be computed is when $\Hg$ corresponds to
the ferromagnetic Ising model and some
extensions~\cite{DemboMontanariIsing},\cite{DemboMontanariSun}.
We note that for the latter models, it suffices that the graph
sequence is left-convergent to some locally tree-like graphs
satisfying some mild additional properties.

Still, the question
whether the two most natural random graph sequences $\G(n,c/n)$
and $\G(n,\Delta)$ are right-convergent (with respect to all
$\Hg$) is still open. The same applies to all other notions of
convergence discussed in this paper (except for
left-convergence). We conjecture that the answer is yes for all
of them. Thanks to Theorem~\ref{theorem:ConvergenceRelations}
this would follow from the following conjecture.
\begin{conj}
The random graph sequences $\G(n,c/n)$ and $\G(n,\Delta)$ are
almost surely LD-convergent.
\end{conj}
The quantifier ''almost sure'' needs some elaboration. Here we
consider the product $\prod_n \mathcal{G}_n$, where
$\mathcal{G}_n$ is the set of all simple graphs on $n$ nodes.
We equip this product space with the natural product
probability measure, arising from having each $\mathcal{G}_n$
induced with random graph $\G(n,c/n)$ measure. Our conjecture
states that, almost surely with respect to this product
probability space $\mathcal{G}$, the graph sequence realization
is LD-convergent.

\section*{Acknowledgements}
We gratefully acknowledge fruitful discussions with Laci
Lov{\'a}sz and wish to thank him for providing the key ideas for
the counterexample in
Subsection~\ref{subsection:PartitionDoesNotImplyLeft}. The
third author also wishes to thank Microsoft Research New
England for providing exciting and hospitable environment
during his visit in 2012.

\bibliographystyle{amsalpha}

\bibliography{LD-bibliography}

\end{document}